%% file: 00-main.tex
\documentclass[12pt]{amsart}

\input{01-packages.tex}
\graphicspath{{figs/}}

\begin{document}

\title{Stationary Surfaces with Boundaries}
\author{Anthony Gruber$^1$, Magdalena Toda$^2$, Hung Tran$^2$}
\address{$^1$\,Department of Mathematics and Statistics, Texas Tech University-Costa Rica, San Jose, 10203, Costa Rica}
\address{$^2$\,Department of Mathematics and Statistics, Texas Tech University, Lubbock, TX 79409, USA}


\input{02-Abstract.tex}

\maketitle
\tableofcontents

\input{03-Intro.tex}

\input{04-Preliminaries.tex}

\input{05-FirstVar.tex}

\input{06-BoundaryCons.tex}

\input{07-RotSym.tex}

\input{08-NoScaleInvar.tex}

\bibliographystyle{abbrv}
\bibliography{Wsurfacebib}

\end{document}

%% file: 01-packages.tex
\usepackage{amsmath,amssymb}
\usepackage{amsfonts}
\usepackage{amsthm}
\usepackage{latexsym}
\usepackage{graphicx}
\usepackage{epsfig}
\usepackage{bm}
\usepackage{mathtools}
\usepackage{enumitem}

\graphicspath{{figures/}}
\usepackage{xifthen}
\usepackage{pdfpages}
\usepackage{transparent}

\newtheorem{theorem}{Theorem}[section]
\newtheorem{lemma}[theorem]{Lemma}
\newtheorem{proposition}[theorem]{Proposition}
\newtheorem{definition}[theorem]{Definition}
\newtheorem{corollary}[theorem]{Corollary}

\newtheorem{prop}{Proposition}[section]
\newtheorem{remark}[prop]{Remark}

\makeatletter \@addtoreset{equation}{section} \makeatother

\def\*#1{\mathbf{#1}}

\def\ddt{\frac{d}{dt}}

\def\He{\mathrm {Hess}\,}
\def\con{{\bm{\eta}}}

\subjclass[2000]{Primary 53A05; Secondary 53A10, 53C40, 53C42}

\date{\today}

\usepackage{xcolor}

%% file: 02-Abstract.tex
\begin{abstract} 
This article investigates stationary surfaces with boundaries, which arise as the critical points of functionals dependent on curvature. Precisely, a generalized ``bending energy'' functional $\mathcal{W}$ is considered which involves a Lagrangian that is symmetric in the principal curvatures. The first variation of $\mathcal{W}$ is computed, and a stress tensor is extracted whose divergence quantifies deviation from $\mathcal{W}$-criticality. Boundary-value problems are then examined, and a characterization of free-boundary $\mathcal{W}$-surfaces with rotational symmetry is given for scaling-invariant $\mathcal{W}$-functionals.  In case the functional is not scaling-invariant,  certain boundary-to-interior consequences are discussed.  Finally, some applications to the conformal Willmore energy and the p-Willmore energy of surfaces are presented.

\vspace{1pc}




\end{abstract}

%% file: 03-Intro.tex
\section{Introduction}
Surfaces with boundaries are fascinating objects which are ubiquitous across mathematics and the natural sciences.  Indeed, many examples of minimal and Willmore surfaces (among others) now serve as idealized models for physically-observable quantities such as surfactant films, lipid membranes, and material interfaces.  Since a large number of relevant surfaces with boundary arise as the minimizers of an energy functional, it is becoming more and more useful to investigate the behavior of these functionals so as to better understand their critical surfaces.

Many significant results reflecting this idea can already be found in the current literature.  In \cite{nitsche1993}, variational problems for surfaces with boundary are studied which involve functionals quadratic in the principal curvatures, and some existence results are proven.  Additionally, \cite{dall2011,dall2012,palmer2000,bergner2010,elliott2017,bernard2018,dalio2020,eichmann2019,eichmann2016,eichmann2018,grunau2021,novaga2020,pozetta2021,schatzle2010} investigate questions of existence and regularity related to boundary-value problems involving the conformally-invariant Willmore functional.  Moreover, related problems involving curvature-dependent energy functionals for surfaces with boundaries have been studied from the perspective of mathematical physics.  In \cite{tu2004,capovilla2002,agrawal2008}, such functionals are used to investigate the elastic properties of lipid membranes, while related functionals are used in \cite{santosa2003} for the analysis and development of lens design. 


Despite the work done so far, much is still unknown regarding the behavior of functionals which depend on surface curvature, especially when the integrand is no longer a quadratic function of the principal curvatures.  To address this, we consider a generalized model, originally proposed by Sophie Germain \cite{germain1831}, for the bending energy of a thin plate.  In particular, if $\*r: \Sigma \to \mathbb{R}^3$ is an isometric immersion of the oriented surface $\Sigma$ with unit normal $\*n:\Sigma\to S^2$ into Euclidean 3-space, the functional of interest will be given as the integral of a symmetric polynomial function in the principal curvatures, which (by a classical theorem of Newton) may be alternatively expressed as (see e.g. \cite{macdonald1998})
\begin{equation}\label{eq:Wfunc}
    \mathcal{W}(\*r) \coloneqq \int_{\Sigma} \tilde{F}(\kappa_1,\kappa_2)\, d\mu = \int_{\Sigma} F(H,K)\,d\mu.
\end{equation}
Here, $H = \kappa_1 + \kappa_2$ and $K = \kappa_1\kappa_2$ are the mean and Gauss curvatures of the surface, respectively, and $d\mu$ is the area element on $\Sigma$ induced by the immersion $\*r$.  Additionally, we allow the possibility of general smooth functions which are symmetric in $\kappa_1$ and $\kappa_2$. 
\begin{remark}
Note that our convention for $H = \kappa_1 + \kappa_2$ is twice the arithmetic mean of the principal curvatures.
\end{remark}

\begin{remark} 
Note that $\mathcal{W}$ reduces to the area functional when $F=1$ and the conformal Willmore functional when $F=H^2-4K$.  Moreover, other functionals of higher-order have been proposed on physical grounds (see \cite{nitsche1993} and references therein), which are also amenable to this formulation.
\end{remark} 

In this article, we study smooth, oriented, compact surfaces with (potentially empty) boundaries. Our motivation is framed by a general question in the calculus of variations, eloquently phrased by B. Palmer in \cite{palmer2000}, which asks whether or not the interior solution to a variational problem necessarily inherits the symmetries of its boundary.  This is natural to consider, as the importance of symmetry in variational problems has been widely-recognized due to a classical theorem of Noether (circa 1918) in \cite{noether1918}.  In particular, Noether's Theorem establishes a valuable correspondence between the symmetries of a Lagrangian (integrand) and the quantities that are conserved under its perturbation.  Among other things, this correspondence encourages the search for divergence-free tensor expressions, often called conservation laws, which encode significant information about the variational problem at hand.  As has been seen in the literature, such expressions can be exceedingly useful in weakening the regularity requirements necessary to prove results (e.g. \cite{riviere2007,bernard2016}).  Additionally, conservation laws have also been used to obtain regularity results in geometrically relevant cases (see e.g. \cite{bernard2018}).  The present contribution to this line of work begins with a first variation formula for $\mathcal{W}$, which is computed in Section 3.  This result, combined with the invariances of $\mathcal{W}$ under translation, rotation, and (when applicable) rescaling leads to flux formulas, which are further used to establish a ``stress tensor'' whose divergence encodes the failure of a surface to be $\mathcal{W}$-critical.  This gives a divergence-form expression of the $\mathcal{W}$-surface Euler-Lagrange equation, which is related to results previously obtained by Y. Bernard and T. Riviere.  More precisely, in \cite{riviere2007} it is shown that all conformally-invariant PDE in 2-dimensions which are non-linear and elliptic admit a divergence-form expression, and in \cite{bernard2018} this computation is extended to more general Euler-Lagrange equations for functionals involving arbitrary functions of the squared mean curvature and the squared norm of the second fundamental form.  A similar expression is presently derived for the Euler-Lagrange PDE characterizing $\mathcal{W}$-surfaces, despite the general lack of conformal invariance in the functional $\mathcal{W}$.

Returning to the question of boundary versus interior inheritance, a partial characterization of rotationally-symmetric $\mathcal{W}$-surfaces with free boundary is given in Section 5.  In this case, it is seen that the answer as to how much the symmetries of the boundary control the solution on the interior is highly dependent on the behavior of the functional $\mathcal{W}$ with respect to rescalings (c.f. Definition~\ref{def:shrinkexp}). Using subscripts to denote partial derivatives with respect to the subscripted quantity, the first main result is as follows.

\begin{theorem}
\label{thm:main1}
	Let $\mathcal{W}$ be scaling invariant, and $\Sigma \subset \mathbb{R}^3$ be an immersed $\mathcal{W}$-surface having free boundary with respect to $\Omega^2 \subset \mathbb{R}^3$. Suppose that $\Sigma$ and $\Omega$ share a common axis of rotational symmetry, and $\Omega$ is strictly convex. Then, one of the following holds:
\begin{enumerate}
\item $\Sigma$ is spherical and $F\equiv 0$ on  $\Sigma$. 
\item $F_H\equiv 0$ and $F_K$ is constant on $\Sigma$. 
\end{enumerate}
\end{theorem}
\begin{remark}
	Either case is certainly possible; see Remark \ref{example1}. For the latter, if $F_K\neq 0$ then, along $\partial \Sigma$, the unit vector of the axis of symmetry is normal to the surface.  
\end{remark}

\begin{remark}
When $\mathcal{W}$ is the conformal Willmore functional, the convexity assumption on $\Omega$ is unnecessary.  See Theorem~\ref{thm:noconvex}.
\end{remark}

This gives a characterization of rotationally-symmetric free-boundary $\mathcal{W}$-surfaces for functionals which are scaling invariant, and also extends what was obtained in \cite{palmer2000} for the conformal Willmore functional.  Moreover, this particular Theorem is seen to hinge on the scaling invariance property of $\mathcal{W}$, as it is not difficult to construct counterexamples when the functional is not scaling invariant (see Remark~\ref{rem:counterex}). This and other applications to conformal Willmore surfaces are discussed further in Section 5. 

On the other hand, there are many interesting $\mathcal{W}$-functionals that do not remain static under rescaling.  For example, the well-known Helfrich-Canham functional \cite{canham1970} for measuring bio-membrane energy per unit area is expressed as
\begin{equation}\label{eq:helfrich}
    \mathcal{W}_{HC}(\*r) := \int_{\Sigma)} k_c(H+c_0)^2 + \overline{k}K\, d\mu,
\end{equation}
where $\overline{k}, k_c$ are some physical rigidity constants, and $c_0$ is known as the spontaneous curvature of the membrane.  From a physical point of view, it is clear that this functional should not be scaling invariant, and indeed it is not.  In fact, knowing that this property does not hold goes a long way toward determining how much control the boundary of a critical surface can exert over the interior.  To study this precisely, we make the following definition
\begin{definition}\label{def:shrinkexp}
A $\mathcal{W}$-functional will be called scaling-invariant provided that 
\begin{equation}\label{eq:scaleinvar}
    F(tH, t^2 K)= t^2 F(H, K),
\end{equation}
for any $t>0$.  On the other hand, $\mathcal{W}$ will be called expanding (resp. shrinking) provided that
\begin{align*}
    2F - H F_H - 2K F_K &\geq 0 \\
    \big(\text{resp.}\,\,2F - H F_H - 2K F_K &\leq 0 \big).
\end{align*} 
In particular, a scaling invariant functional is both expanding and shrinking.
\end{definition}

With this terminology in place, the following is proved.

\begin{theorem}\label{thm:main2}
Suppose $\mathcal{W}$ is a functional which is either shrinking or expanding, and let $\Sigma \subset \mathbb{R}^3$ be an immersed $\mathcal{W}$-surface with boundary $\partial\Sigma$ and adapted orthonormal frame field $\{\*T,\*n,\con\}$ such that $\*T$ is tangent to $\partial\Sigma$, $\*n$ is everywhere normal to $\Sigma$, and $\con = \*T\times \*n$. Suppose additionally that the following boundary conditions are satisfied:
 \begin{align*}
     0 &= \tau_g F_H, \\
     0 &= F - h(\con,\con)F_H -K F_K, \\
     0 &= h(\nabla F_K,\con) - \nabla_\con F_H - H\nabla_\con F_K, 
 \end{align*}
where $h:T\Sigma\times T\Sigma\to\mathbb{R}$ denotes the second fundamental form of $\Sigma$ and $\tau_g = h(\*T,\*\con)$ is the geodesic torsion of $\partial\Sigma$.  Then, $2F-HF_H-2KF_K\equiv 0$ on $\Sigma$.
\end{theorem}

This result shows that the interior behavior of dilation-sensitive $\mathcal{W}$-functionals is highly affected by conditions on the boundary, and suggests a partial explanation for the differences seen between boundary-value problems for the conformal Willmore functional when compared to those for more rigid $\mathcal{W}$-functionals like the Helfrich-Canham energy (c.f. \cite{palmer2000},\cite{tu2004},\cite{elliott2017},\cite{bohle2008}).  The consequences of this are discussed further in Section 6.  One particularly interesting application involves the p-Willmore energy functional discussed in \cite{gruber2019,mondino2014},
\begin{equation*}
    \int_\Sigma |H|^p\, d\mu, \qquad p\in\mathbb{R}.
\end{equation*}
Consideration of Theorem~\ref{thm:main2} shows that in this case, for some values of $p$, there are no non-minimal critical surfaces which have zero mean curvature on their boundary.  More precisely, the following is observed.
\begin{theorem}\label{thm:pWillmore}
When $p>2$, any p-Willmore surface $\Sigma\subset\mathbb{R}^3$ with boundary which satisfies $H=0$ on $\partial\Sigma$ must be a minimal surface.
\end{theorem}

\begin{remark}
Note that it is possible to deduce from Theorem~\ref{thm:pWillmore} that there are no closed p-Willmore surfaces $\Sigma$ immersed in $\mathbb{R}^3$ when $p>2$.  Indeed, equation the result asserts that any such surface must be minimal, but there are no closed minimal surfaces immersed in $\mathbb{R}^3$, a contradiction.  This recovers a result proved in \cite{gruber2019}.
\end{remark}


To summarize, this manuscript is structured as follows: Section 2 briefly recalls the necessary mathematical background; Section 3 demonstrates the first variation of (\ref{eq:Wfunc}) and collects its ramifications; Section 4 considers how different conditions on the boundary influence the critical surfaces of $\mathcal{W}$-functionals; Section 5 studies such boundary-value problems subject to a rotational symmetry constraint and establishes Theorem~\ref{thm:main1}; Section 6 examines these problems for $\mathcal{W}$-functionals that are not scaling invariant, and establishes Theorems~\ref{thm:main2} and \ref{thm:pWillmore}.

\begin{remark}
It is interesting to note that many of the Theorems established here should have an analogue for hypersurfaces in any dimension, as the notion (but not the expression) of scaling invariance and the existence of $\*n$ are independent of these notions.  On the other hand, it is unlikely that these results extend to arbitrary codimension in their present form without some assumptions on the normal bundle to the immersion (e.g. flatness, parallel mean curvature).  Since a general curvature functional of interest $\int_\Sigma F(\kappa_1,...,\kappa_n)\,d\mu$ is also more complicated in dimension $n$, the results presented here leave plenty of compelling questions for future work.
\end{remark}
{\bf Acknowledgment.} Hung Tran was partially supported by a Simons Foundation Collaboration Grant and NSF grant DMS-2104988.  Magdalena Toda was partially supported by Simons Foundation Collaboration grant number 632274.  Part of this work was done while Hung Tran was visiting the Vietnam Institute for Advanced Study in Mathematics (VIASM). He would like to thank VIASM for financial support and hospitality.  All authors would like to thank the anonymous reviewers for helpful comments.

%% file: 04-Preliminaries.tex
\section{Preliminaries}
In this section, we will fix the notation and conventions that will be used throughout the paper, and collect variation formulas for a surface in Euclidean space. First, let $\Sigma$ be a smooth oriented surface with potential boundaries and let $\*r$ be an isometric immersion, 
\[ \*r: \Sigma^2 \mapsto \mathbb{R}^{3},\]
with choice of unit normal field $\*n$.  Let $g$ be the metric on $\Sigma$ induced from the standard metric on $\mathbb{R}^3$, and let $d\mu$ denote its associated volume form. Let $D, \nabla$ be the connections on $\mathbb{R}^3$ and $(\Sigma,g)$ respectively. The second fundamental form $h$, mean curvature $H$, and Gaussian curvature $K$ are then defined as follows (Einstein summation assumed). For orthonormal vector fields $\*e_i, \*e_j \in TM$, it follows that
\begin{align*}
h_{ij} &=\left\langle{D_{\*e_i}\*e_j, \*n}\right\rangle=-\left\langle{D_{\*e_i}\*n, \*e_j}\right\rangle;\\
H&= g^{ij}h_{ij};\\
K &= \text{det}\left(g^{ik}h_{kj}\right) = \frac{\text{det}\,h}{\text{det}\,g},
\end{align*}
where $\langle\cdot,\cdot\rangle$ denotes the standard inner product on $\mathbb{R}^3$.
Moreover, along boundaries $\partial \Sigma$ we let $\con$ be the outward co-normal unit vector, $\*T$ the unit tangential vector field, and $ds$ the associated arc length.  As a consequence, the curvatures on the boundary can be expressed as,
\begin{align*}
    &H = h(\*T,\*T) + h(\con,\con) = \kappa_n + h(\con,\con), \\
    &K = h(\*T,\*T)h(\con,\con) - h(\*T,\con)^2 = h(\con,\con)\kappa_n - \tau_g^2,
\end{align*}
where the quantities
\begin{align*}
    \kappa_n &= \langle \nabla_\*T \*T, \*n \rangle = h(\*T,\*T), \\
    \tau_g &= \langle \nabla_\*T \con, \*n\rangle = h(\*T,\con)
\end{align*}
are, respectively, the normal curvature and geodesic torsion of $\partial\Sigma$ when considered as a curve in $\Sigma$.  Note that $\kappa_n$ measures how fast $\*T$ rotates into $\*n$ along the boundary curve, while $\tau_g$ measures how fast $\con$ rotates into $\*n$ (c.f. Figure~\ref{fig:boundary}).

\begin{figure}[htb!]
    \centering
    \def\svgwidth{0.5\linewidth}
    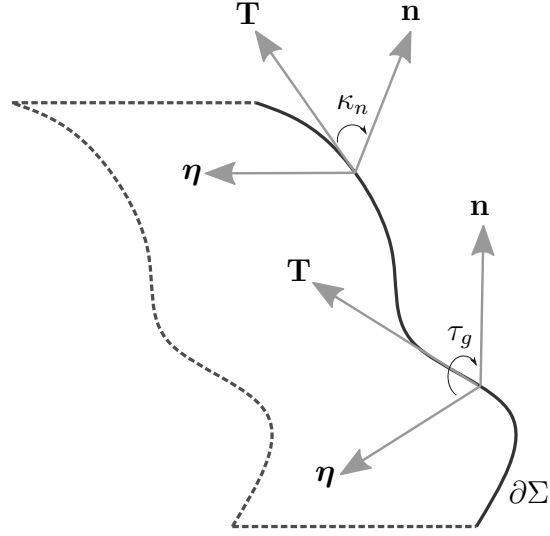
    \caption{$\kappa_n$ and $\tau_g$ as rates of rotation (for a left-handed frame).}
    \label{fig:boundary}
\end{figure}

\subsection{Functionals of interest}
The primary objects of study in this work will be functionals of the form
\[ \mathcal{W}(\*r) := \int_{\Sigma} \widetilde{F}\, d\mu,\]
where $\widetilde{F}$ is a smooth symmetric function in the principal curvatures of $\*r(\Sigma)\subset \mathbb{R}^3$. Expressed differently, this implies that $\widetilde{F}=F(H, K)$ is a smooth function of $H$ and $K$.

As mentioned in the Introduction, it is useful to keep in mind the symmetries that are present.  It is clear that a general $\mathcal{W}$-functional is invariant under translations and rotations of $\mathbb{R}^3$ (since $H$ and $K$ are rigid-motion invariant).  However, any particular $\mathcal{W}$-functional need not be invariant under changes of scale.  To see this, recall the consequences of rescaling an immersion $\*r \mapsto (1/t)\*r$ by some $t>0$.  In particular,
\begin{align*}
g  &\mapsto \frac{1}{t^2}\, g,\\
d\mu &\mapsto \frac{1}{t^2}\, d\mu,\\
H &\mapsto tH,\\
K &\mapsto t^2K, \\
\Delta &\mapsto t^2\Delta
\end{align*}
From this, it follows that the derivative of $\mathcal{W}$ under rescaling satisfies
\[\frac{d}{dt}\bigg|_{t=1}\int_\Sigma F(tH,t^2K) \frac{1}{t^2}\,d\mu = \int_\Sigma (H F_H + 2 K F_K - 2F)\,d\mu,\] 
which forms the motivation for Definition~\ref{def:shrinkexp}.  Note that a functional is either strictly expanding or shrinking if and only if the scaling excess $2F - H F_H - 2K F_K$ is strictly positive or negative.  Moreover, consideration of equation (\ref{eq:scaleinvar}) immediately yields the relationships
\begin{align*}
F_H (tH, t^2K) &= tF_H(H,K),\\
F_K (tH, t^2K) &= F_K(H,K),\\
F(0,0) &= F_H(0,0) = 0.
\end{align*} 
In particular, $F_K$ is itself scaling invariant whenever $\mathcal{W}$ is.


\subsection{Variation of geometric quantities}\label{variationpreliminaries}
It is advantageous to collect the various evolution equations that will be needed for the analysis of $\mathcal{W}$-functionals.  To that end, consider a variation of the immersion $\*r$ by a velocity vector field $\*X=u\*n+\bm{\zeta}$ where $u$ is smooth on $\Sigma$, $\*n$ is a choice of unit normal, and $\bm{\zeta}$ is tangential to the surface:
\begin{equation}\label{eq:normalvar}
\delta_\*X \*r \coloneqq \ddt \*r\Bigr|_{t=0} = \*X.
\end{equation}
There are then the following well known normal evolution equations; for example, see \cite{gruber2019}. 
\begin{align*}
\delta_{u\*n} g &= -2u h,\\
\delta_{u\*n} g^{ij} &= 2u h_{ij},\\
\delta_{u\*n} h_{ij} &= (\He{u})_{ij}-uh_{i}^{\ell}h_{\ell j},\\
\delta_{u\*n}  d\mu &= -uH \, d\mu,\\
\delta_{u\*n}  |h|^2 &= 2\left\langle{h, \He u}\right\rangle+2u |h|^3,\\
\delta_{u\*n}  &= -\nabla u,\\
\delta_{u\*n}  H &= u |h|^2+\Delta u,\\
\delta_{u\*n}  K &= H\Delta u-\left\langle{h, \He u}\right\rangle+ HKu,
\end{align*}
where $\He{u}$ denotes the Hessian of $u$ and $|h|^2 = H^2-2K$ denotes the squared norm of the second fundamental form.  Moreover, the surface Laplacian evolves by the following equation,
\begin{equation*}
\resizebox{0.99\hsize}{!}{%
$\delta_{u\*n} (\Delta f)= 2u \left\langle{h, \He f}\right\rangle+\Delta \left(\ddt f\right)+2h(\nabla u, \nabla f)- H\langle\nabla u, \nabla f\rangle +u \langle\nabla H,\nabla f\rangle.$
}
\end{equation*}

In addition, the variation induced by the tangential vector field $\bm{\zeta}$ is tracked as a Lie derivative. That is, for any function $f$, we have
\begin{equation*}
\delta_{\bm{\zeta}} \int_\Sigma f\, d\mu =\int_\Sigma \mathcal{L}_{\bm{\zeta}} \left(f\, d\mu\right) = \int_{\partial \Sigma} f\left\langle{\bm{\zeta}, \con}\right\rangle\, ds,
\end{equation*}
where the final equality is due to Stokes' Theorem and the fact that $df \wedge d\mu \equiv 0$ on $\Sigma$, since $\mu$ is a volume form.

%% file: figs/boundary.pdf_tex
\begingroup%
  \makeatletter%
  \providecommand\color[2][]{%
    \errmessage{(Inkscape) Color is used for the text in Inkscape, but the package 'color.sty' is not loaded}%
    \renewcommand\color[2][]{}%
  }%
  \providecommand\transparent[1]{%
    \errmessage{(Inkscape) Transparency is used (non-zero) for the text in Inkscape, but the package 'transparent.sty' is not loaded}%
    \renewcommand\transparent[1]{}%
  }%
  \providecommand\rotatebox[2]{#2}%
  \ifx\svgwidth\undefined%
    \setlength{\unitlength}{399.10350037bp}%
    \ifx\svgscale\undefined%
      \relax%
    \else%
      \setlength{\unitlength}{\unitlength * \real{\svgscale}}%
    \fi%
  \else%
    \setlength{\unitlength}{\svgwidth}%
  \fi%
  \global\let\svgwidth\undefined%
  \global\let\svgscale\undefined%
  \makeatother%
  \begin{picture}(1,0.94480601)%
    \put(0,0){\includegraphics[width=\unitlength,page=1]{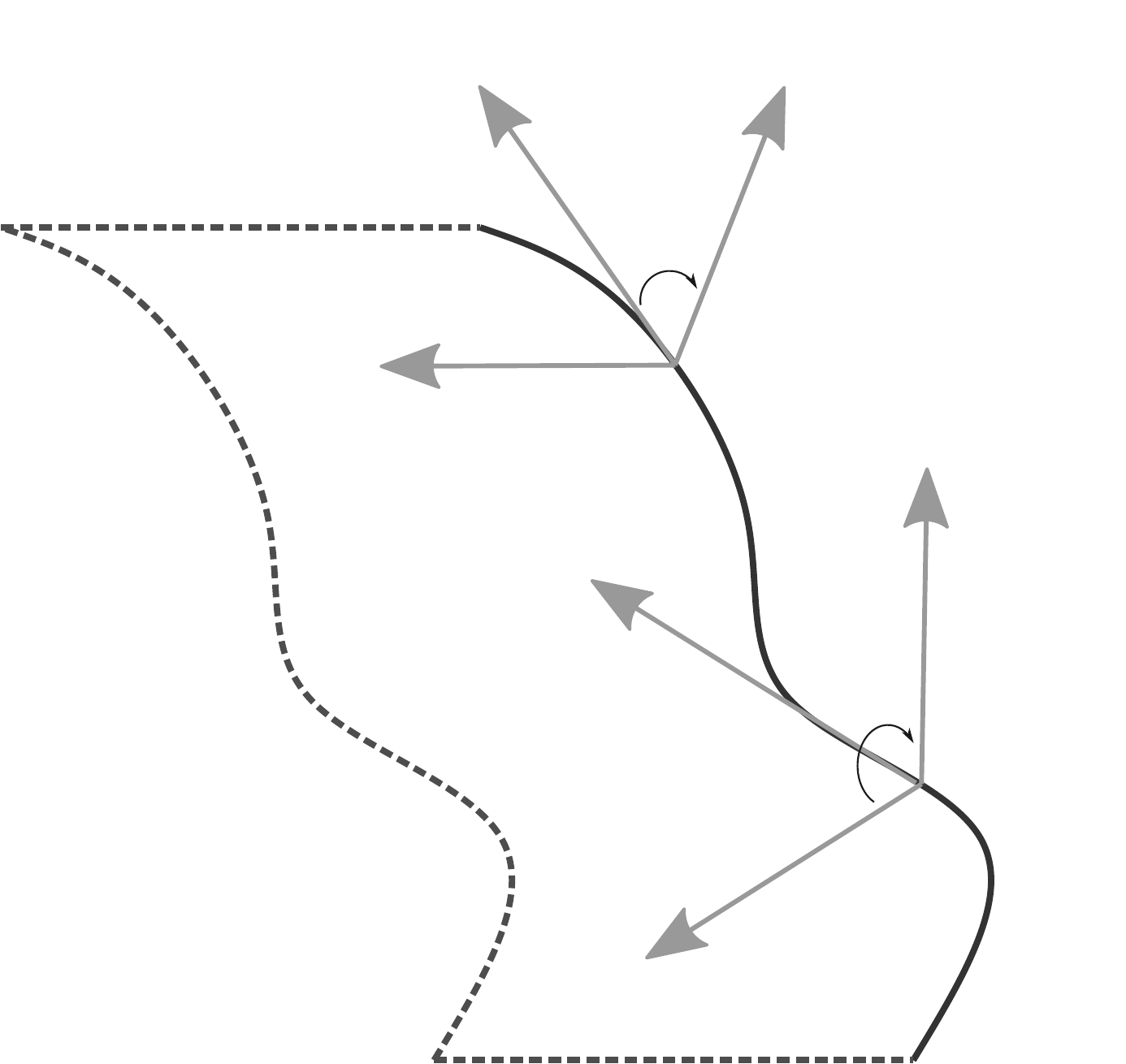}}%
    \put(0.39083435,0.87939057){\color[rgb]{0,0,0}\makebox(0,0)[lb]{\smash{$\*T$}}}%
    \put(0.67943378,0.88769559){\color[rgb]{0,0,0}\makebox(0,0)[lb]{\smash{$\*n$}}}%
    \put(0.29740287,0.60947735){\color[rgb]{0,0,0}\makebox(0,0)[lb]{\smash{$\con$}}}%
    \put(1.03239721,0.58248607){\color[rgb]{0,0,0}\makebox(0,0)[lb]{\smash{}}}%
    \put(0.79985663,0.54718972){\color[rgb]{0,0,0}\makebox(0,0)[lb]{\smash{$\*n$}}}%
    \put(0.48011329,0.4309194){\color[rgb]{0,0,0}\makebox(0,0)[lb]{\smash{$\*T$}}}%
    \put(0.52994342,0.08210849){\color[rgb]{0,0,0}\makebox(0,0)[lb]{\smash{$\con$}}}%
    \put(0.7604078,0.32918292){\color[rgb]{0,0,0}\makebox(0,0)[lb]{\smash{$\tau_g$}}}%
    \put(0.56731598,0.73197646){\color[rgb]{0,0,0}\makebox(0,0)[lb]{\smash{$\kappa_n$}}}%
    \put(0.86787477,0.04360133){\color[rgb]{0,0,0}\makebox(0,0)[lb]{\smash{$\partial\Sigma$}}}%
  \end{picture}%
\endgroup%

%% file: 05-FirstVar.tex
\section{First Variation}
It is now opportune to compute the first variation formula corresponding to  (\ref{eq:Wfunc}), which will facilitate the analysis of $\mathcal{W}$-functionals.  In particular, the formulation presented here is applicable to both closed surfaces as well as surfaces with nontrivial boundary.  Further, symmetries of the $\mathcal{W}$-surface variational problem will be used to generate flux formulas, and a stress tensor will be given whose divergence measures the deviation of a surface from $\mathcal{W}$-criticality.  
\begin{theorem}\label{thm:firstvar}
Let $\Sigma$ be a compact smooth surface and $\*r(t): \Sigma \mapsto \mathbb{R}^3$ be a family of diffeomorphisms with velocity 
\[\delta_\*X \*r = \ddt \*r \Bigr|_{t=0} = \*X. \]
Then, the first variation of the functional $\mathcal{W}$ is given by
\begin{align*}
\delta_{\*X} &\mathcal{W} = \int_{\partial \Sigma} F\left\langle{\*X, \con}\right\rangle\, ds + \int_{\partial \Sigma} \left\langle{\*X, \*n}\right\rangle\Big(h(\nabla F_K, \con)-\nabla_\con F_H-H\nabla_\con F_K\Big)\, ds \\
&+\int_{\partial \Sigma} \Big((F_H+HF_K)\nabla_\con \left\langle{\*X, \*n}\right\rangle -F_K h(\nabla \left\langle{\*X, \*n}\right\rangle, \con)\Big)\, ds\\
&+ \int_{\Sigma} \left\langle{\*X, \*n}\right\rangle\Big(\Delta F_H+H\Delta F_K-\left\langle{h, \He F_K}\right\rangle+F_H|h|^2+HKF_K-HF\Big)\, d\mu.\\
\end{align*}
\end{theorem}

\begin{proof}
For $X=u\*n+\bm{\zeta}$, the formulas from Section \ref{variationpreliminaries} yield
\begin{align*}
\delta_{u\*n}& \mathcal{W} =\int_{\Sigma} \Big(F_H (\delta_{u\*n}H)+F_K (\delta_{u\*n}K)\Big)\,d\mu+\int_{\Sigma}F (\delta_{u\*n} \,d\mu)\\
&\hspace{-0.3pc}= \int_{\Sigma}\Big(F_H(u|h|^2+\Delta u)+F_K(H\Delta u-\left\langle{h, \He u}\right\rangle +HKu)-uHF\Big)\, d\mu\\
&\hspace{-0.3pc}= \int_{\Sigma}\Big((F_H+HF_K)\Delta u +(F_H|h|^2+HKF_K-HF)u -F_K\left\langle{h, \He u}\right\rangle\Big)\, d\mu.
\end{align*}
Moreover, it follows from integration by parts and the Codazzi equation $\mathrm{div}_g\, h = \nabla H$ that
\begin{align*}
\int_{\Sigma} (F_H+HF_K)\Delta u \,d\mu &= \int_{\Sigma} u\Delta(F_H+HF_K)\, d\mu \\
&+\int_{\partial \Sigma} \Big((F_H+HF_K)\nabla_\con u -u\nabla_\con(F_H+HF_K)\Big)\,ds,\\
\int_{\Sigma} F_K\left\langle{h, \He u}\right\rangle d\mu &= \int_\Sigma u(\left\langle{h, \He F_K}\right\rangle + 2\langle\nabla F_K, \nabla H\rangle + F_K\Delta H)\, d\mu\\
&+\int_{\partial \Sigma} \Big(F_K h(\nabla u, \con)-u h(\nabla F_K, \con)-uF_K \nabla_\con H \Big)\, ds.
\end{align*}
Putting the above expressions together, we obtain
\begin{align*}
\delta_{u\*n} \mathcal{W} &=\int_{\Sigma} u\Big(\Delta(F_H+HF_K)-\left\langle{h, \He F_K}\right\rangle - 2\nabla F_K \nabla H - F_K\Delta H\Big)\,d\mu \\
&+\int_{\Sigma} u(F_H|h|^2+HKF_K-HF)\, d\mu\\
&+\int_{\partial \Sigma} \Big((F_H+HF_K)\nabla_\con u -u\nabla_\con(F_H+HF_K)\Big) \,ds\\
&+\int_{\partial \Sigma} \Big(-F_K h(\nabla u, \con)+u h(\nabla F_K, \con)+ uF_K \nabla_\con H \Big) \,ds,\\
&= \int_{\Sigma} u\Big(\Delta F_H+H\Delta F_K-\left\langle{h, \He F_K}\right\rangle + F_H|h|^2+HKF_K-HF\Big) \,d\mu\\
&+\int_{\partial \Sigma} \Big((F_H+HF_K)\nabla_\con u -F_K h(\nabla u, \con)\Big) \, ds\\
&+\int_{\partial \Sigma} u\Big(h(\nabla F_K, \con)-\nabla_\con F_H-H\nabla_\con F_K\Big) \, ds.
\end{align*}
Finally, recall that the tangential variation can be computed as
\[\delta_{\bm{\zeta}}\mathcal{W} = \delta_{\bm{\zeta}} \int_\Sigma F \,d\mu = \int_{\partial \Sigma} F\left\langle{\bm{\zeta}, \con}\right\rangle \, ds.\]
The result then follows. 
\end{proof}

\begin{remark}
In the case $F=H^2-4K$ corresponding to the conformally-invariant Willmore functional, we immediately recover the following (c.f. \cite{palmer2000}),
	\begin{align*}
	\delta_\*X \int_M (H^2-4K)\, d\mu &= \int_\Sigma \left\langle{\*X, \*n}\right\rangle (2\Delta H+ H(|h|^2-2K)) \,d\mu\\
	&+\int_{\partial \Sigma} \Big(4h-2 Hg\Big)(\nabla \left\langle{\*X, \*n}\right\rangle, \con)\, ds \\
	&+\int_{\partial \Sigma} \Big((H^2-4K) \left\langle{\*X, \con}\right\rangle-2\left\langle{\*X, \*n}\right\rangle\left\langle{\con,\nabla H}\right\rangle\Big)\, ds.
	\end{align*}
\end{remark}

The results of Theorem~\ref{thm:firstvar} motivate the following definition. 
\begin{definition}
$\Sigma$ is said to be a stationary surface with respect to $\mathcal{W}$ (or, in short, a $\mathcal{W}$-surface) provided it satisfies the Euler-Lagrange equation
 \[\Delta F_H+H\Delta F_K-\left\langle{h, \He F_K}\right\rangle+F_H|h|^2+HKF_K-HF=0.\]
\end{definition}

The first variation above immediately leads to some  useful flux formulas. 
\begin{corollary}\label{cor:fluxformula}
  Let $\Sigma$ be a compact $\mathcal{W}$-surface with boundary and $\*e$ a constant vector field. Then, the following hold:
\begin{align}
0 &= \int_{\partial \Sigma} \Big((F_H+HF_K)\nabla_\con \left\langle{\*e, \*n}\right\rangle -F_K h(\nabla \left\langle{\*e, \*n}\right\rangle, \con)\Big)\, ds \label{eq:transflux} \\
&+\int_{\partial \Sigma} \left\langle{\*e, \*n}\right\rangle\Big(h(\nabla F_K, \con)-\nabla_\con F_H-H\nabla_\con F_K\Big) \,ds + \int_{\partial \Sigma} F\left\langle{\*e, \con}\right\rangle \,ds; \notag \\
\int_\Sigma &(2F-HF_H-2KF_K)\,d\mu \label{eq:scaleflux} \\
&= \int_{\partial \Sigma} \Big((F_H+HF_K)\nabla_\con \left\langle{\*r, \*n}\right\rangle -F_K h(\nabla \left\langle{\*r, \*n}\right\rangle, \con)\Big)\, ds \notag\\
&+\int_{\partial \Sigma} \left\langle{\*r, \*n}\right\rangle\Big(h(\nabla F_K, \con)-\nabla_\con F_H-H\nabla_\con F_K\Big)\, ds+\int_{\partial \Sigma} F\left\langle{\*r, \con}\right\rangle\, ds; \notag \\
0 &= \int_{\partial \Sigma} \Big((F_H+HF_K)\nabla_\con \left\langle{\*e\times \*r, \*n}\right\rangle -F_K h(\nabla \left\langle{\*e\times \*r, \*n}\right\rangle, \con)\Big)\, ds \label{eq:rotflux} \\
&+\int_{\partial \Sigma} \left\langle{\*e\times \*r, \*n}\right\rangle\Big(h(\nabla F_K, \con)-\nabla_\con F_H-H\nabla_\con F_K\Big)\, ds+\int_{\partial \Sigma} F\left\langle{\*e\times \*r, \con}\right\rangle \,ds \notag.
\end{align}
\end{corollary}
\begin{proof}
First, consider a continuous family of translations $\*r(t)=\*r+t\*e$ for $(-\epsilon\leq t\leq \epsilon)$. It follows that
\begin{align*}
\ddt \*r(t)\Bigr|_{t=0} &= \*e, \\
\mathcal{W}(\*r(t)) &= \mathcal{W}(\*r(0)).
\end{align*}
The first expression now follows from Theorem \ref{thm:firstvar}.  Next, consider a continuous family of rescalings $\*r(t)=t\*r$ for $(1-\epsilon\leq t\leq 1+\epsilon)$. Then, it follows that 
\begin{align*}
\ddt \*r(t)\Bigr|_{t=1} &= \*r,
\end{align*}	
Moreover, the consequences of rescaling recalled in Section 2 imply that
\[\mathcal{W}(\*r(t)) = \int_{\Sigma} F\left(\frac{H}{t}, \frac{K}{t^2}\right) t^2\, d\mu.\]
Taking the derivative at $t=1$ and applying Theorem \ref{thm:firstvar} now yields the second expression. Finally, consider a continuous family of rotations around a unit constant vector $\*e \in S^2$. By Rodrigues' rotation formula, it follows that
\begin{align*}
\*r(t) &= \*r\cos(t)+ (\*e\times \*r) \sin(t)+\left\langle{\*r, \*e}\right\rangle(1-\cos(t))\*e,\\
\ddt \*r(t)\Bigr|_{t=0} &= \*e\times \*r,\\
\mathcal{W}(\*r(t)) &= \mathcal{W}(\*r(0)),
\end{align*}
where $\times$ denotes the standard right-handed cross product on $\mathbb{R}^3$.  Again, applying Theorem \ref{thm:firstvar} leads to the third identity. 
\end{proof}

These flux formulas are useful to examine in the broader context of conservation laws.  To that end, recall the usual shape operator $S: T\Sigma \to T\Sigma$ defined by 
\[ \left\langle S(\*v),\*w \right\rangle = \left\langle -\nabla_\*v \*n, \*w \right\rangle = h(\*v,\*w), \]
for all vector fields $\*v,\*w \subset T\Sigma$, and recall that $S$ is known to be a linear map which is  self-adjoint with respect to the metric inner product on $\Sigma$ \cite[Chapter 13]{gray2006}.  Since the Euclidean inner product $\langle \cdot,\cdot \rangle$ on $\mathbb{R}^3$ restricts to give the metric inner product on $\Sigma \subset \mathbb{R}^3$, it follows that
\[ \left\langle S(\*v), \*w \right\rangle = \left\langle \*v, S(\*w) \right\rangle, \]
for all $\*v,\*w \subset T\Sigma.$ As a consequence of this, note that 
\[ \left\langle S^2(\*v), \*w \right\rangle = \left\langle S\left(S(\*v)\right), \*w \right\rangle = \left\langle S(\*v), S(\*w) \right\rangle = \left\langle \*v, S^2(\*w) \right\rangle. \]
Moreover, since $\nabla_\*v \*n \subset T\Sigma$ for all $\*v \subset T\Sigma$, it is evident that any ambient vector field $\*e \subset T\mathbb{R}^3$ satisfies
\[\left\langle \*e, \nabla_\*v \*n\right\rangle = \left\langle \*e^\top, \nabla_\*v \*n \right\rangle, \]
where $\*e^\top$ denotes the projection of $\*e$ onto $T\Sigma$.  In view of this, $S(\*e)$ will be used to denote the vector $S\left(\*e^\top\right)$ in the sequel.

With these additional notions in place, it is now possible to construct a stress tensor associated to the $\mathcal{W}$-functional whose divergence encodes deviation from $\mathcal{W}$-criticality.  This implies a conservation law for $\mathcal{W}$-surfaces as expressed by the following result.

\begin{theorem}\label{thm:stress}
Let 
\begin{align*}
&T = F_K\,S^2 - (F_H + H F_K)\,S + \left( S(\nabla F_K) - \nabla F_H - H\nabla F_K\right) \otimes  \*n + F\nabla\*r, \\
&W = \Delta F_H+H\Delta F_K-\left\langle{h, \He F_K}\right\rangle+F_H|h|^2+HKF_K-HF.
\end{align*}
Then, it follows that \[\mathrm{div}_g\, T = -W\*n.\]  In particular, $\Sigma$ is a $\mathcal{W}$-surface if and only if $T$ is divergence-free.
\end{theorem}

\begin{proof}
Though this can be verified by direct computation, it is more instructive to derive this result as a consequence of translation invariance and Theorem~\ref{thm:firstvar}.  First, note that $\nabla_\*v\*r = \*v$ for any tangent vector $\*v$.  Moreover, let  $\*e$ be a constant vector field. Then, it follows from the definition of $S$ and the discussion above that
\begin{equation*}
\begin{split}
    &\int_{\partial\Sigma} (F_H + H F_K)\langle \*e,\nabla_\con \*n\rangle - F_K \, h\left(\nabla_\*e \*n, \con\right) \, ds \\
    &+\int_{\partial\Sigma} \langle\*e,\*n\rangle \left( h(\nabla F_K, \con) - \nabla_\con F_H - H\nabla_\con F_K \right) + F\langle \*e, \con \rangle\, ds \\
    &= \int_{\partial\Sigma} \left\langle \big( -(F_H + H F_K)\, S(\*e) + F_K \, S^2(\*e), \con\right\rangle\,ds \\
    &+\int_{\partial\Sigma} \left\langle \left(S(\nabla F_K) - \nabla F_H - H\nabla F_K\right) \langle\*e, \*n \rangle + F\nabla_\*e\*r \big), \con\right\rangle\, ds.
\end{split}
\end{equation*}

Moreover, using integration by parts, the above can be expressed as
\begin{equation*}
    \int_{\partial\Sigma} \langle T\*e,\con\rangle\, ds = \int_{\Sigma}\mathrm{div}_g\, (T\*e)\, d\mu = \int_\Sigma \left\langle \*e, \mathrm{div}_g\, T \right\rangle \, d\mu,
\end{equation*}
where $\mathrm{div}_g$ denotes the divergence with respect to the metric $g$ and the constancy of $\*e$ was used in the last equality.  Translation invariance and Theorem~\ref{thm:firstvar} now imply that for any constant vector field $\*e$ and any surface $\Sigma$,
\begin{equation} \label{eq:forconservation}
    0 = \delta_\*e \mathcal{W} = \int_\Sigma \langle \*e, W\*n + \mathrm{div}_g\, T\rangle\, d\mu.
\end{equation}
To complete the argument, we claim that the above implies that 
\[W\*n + \mathrm{div}_g\, T \equiv \*0.\]
To verify this, suppose it is not true.  Then, there must be a constant vector field $\*e_0$ such that (\ref{eq:forconservation}) is true for all $\Sigma$ but $W\*n + \mathrm{div}_g\, T$ is nonzero.  First, notice that $\*e_0$ cannot be everywhere orthogonal to $W\*n + \mathrm{div}_g\, T$ in this case, since the latter field is not constant.  Indeed, if it were, choosing $\*{e}_0 = W\*n + \mathrm{div}_g\, T$ in (\ref{eq:forconservation}) would produce a contradiction.   Moreover, smoothness implies that the function $\left\langle \*e_0, W\*n + \mathrm{div}_g\, T\right\rangle$ varies continuously on any $\Sigma$, so for any $p\in\Sigma\setminus\partial\Sigma$ where $W\*n + \mathrm{div}_g\, T \neq \*0$ we may choose a local surface $\Sigma_0 \subset \Sigma$ containing $p$ on which this field is strictly positive or strictly negative.  Without loss of generality, suppose that $\left\langle \*e_0, W\*n + \mathrm{div}_g\, T\right\rangle > 0$ on $\Sigma_0$.  In this case,
\[\int_{\Sigma_0} \left\langle \*e_0, W\*n + \mathrm{div}_g\, T\right\rangle\,d\mu > 0\]
which contradicts equation (\ref{eq:forconservation}).  Hence, the claim is true and the result follows.
\end{proof}

\begin{remark}\label{rem:recovery}
In the case where $F = H^2$, the stress tensor above reduces to 
\[T = -2H\, S - 2 \nabla H \otimes \*n + H^2\nabla \*r,\]
which coincides with the expression in \cite{bernard2016} up to our convention for $H$.  Moreover, considering the special cases $F = F(H^2)$ and $F = F(|h|^2)$ recovers Proposition 2.1 in \cite{bernard2018}.
\end{remark}

Besides their pleasing physical interpretation as conservation laws, divergence-form expressions have been historically helpful for the study of problems involving harmonic maps, minimal surfaces, and Willmore immersions.  Particularly, the ability to suppress one derivative in the Euler-Lagrange equation has enabled researchers to prove interesting results under much lighter regularity requirements than would otherwise be possible (see e.g. \cite{riviere2007} and the references therein). As mentioned in the Introduction, the result of Theorem~\ref{thm:stress} is similar in nature to a result of T. Rivi{\`e}re which asserts that all conformally-invariant PDE in 2 dimensions which are non-linear and elliptic admit a divergence-form expression.  Indeed, when the conformally-invariant PDE in question comes from a $\mathcal{W}$-functional, this result is recovered from Theorem~\ref{thm:stress} (see e.g. Remark~\ref{rem:recovery}).

%% file: 06-BoundaryCons.tex
\section{Boundary Considerations}
To develop knowledge about surfaces with boundaries, critical points of a generic functional $\mathcal{W}$ are now studied subject to different conditions at the boundary. First, the fixed-boundary condition is investigated. Note that in this case all variations must vanish along $\partial\Sigma$.
\begin{proposition}\label{prop:fixedboundary}
Let $\*r:\Sigma\mapsto M\subset \mathbb{R}^3$ be a smooth isometric immersion of a surface with boundary. Then $\Sigma$ is a critical point of $\mathcal{W}$ under the fixed-boundary condition if and only if,
\begin{align*}
0 &= \Delta F_H + H\Delta F_K-\left\langle{h, \He F_K}\right\rangle+F_H|h|^2+HKF_K-HF, \text{~~in $\Sigma$}\\
0 &= F_H + \kappa_nF_K \text{~~ on $\partial \Sigma$}.
\end{align*}
\end{proposition}
\begin{proof}
Let $\*X$ be a velocity vector field on $\Sigma$. It follows from the fixed-boundary condition that $\*X\equiv \*0$ on $\partial \Sigma$. Furthermore, $\nabla_\*T\left\langle{\*X,\*n}\right\rangle=0$ on $\partial \Sigma$. Thus, Theorem \ref{thm:firstvar} implies, 
\begin{align*}
\delta_\*X &\mathcal{W} = \int_{\partial \Sigma} \Big((F_H+HF_K-F_Kh(\con,\con)\Big)\nabla_\con \left\langle{\*X, \*n}\right\rangle\, ds \\
&+ \int_{\Sigma} \left\langle{\*X, \*n}\right\rangle\Big(\Delta F_H+ H\Delta F_K-\left\langle{h, \He F_K}\right\rangle+F_H|h|^2+HKF_K-HF\Big) \, d\mu.
\end{align*}
Note that $\Sigma$ is a critical point if and only if $\delta_\*X\mathcal{W} =0$ for all such $\*X$. The first equation follows since it is possible to choose $\left\langle{\*X,\*n}\right\rangle$ to be zero everywhere except for any arbitrary small interior disk.  Then, as $\Sigma$ is smooth, elliptic theory (see e.g. \cite{gilbarg2001}) allows one to solve the biharmonic Dirichlet-Neumann problem 
\begin{align*}
\Delta^2 u &=0 \text{~~ in $\Sigma$},\\
u &=0 \text{~~ on $\partial \Sigma$},\\
\nabla_\con u &= F_H+(H-h(\con,\con))F_K \text{~~ on $\partial \Sigma$},
\end{align*} 
where $\*X = u\*n$ was chosen.  The second equation now follows since $H - h(\con,\con) = \kappa_n$ on $\partial\Sigma$. 
\end{proof}

On the other hand, it is reasonable to consider the possibility of a surface with free boundary, as is the case in many applications. That is, let $\Omega$ be a smooth surface in $\mathbb{R}^3$, and consider all variations of $\*r:\Sigma \mapsto \mathbb{R}^3$ such that $\*r(\partial \Sigma)\subset \Omega$. The following result characterizes $\mathcal{W}$-surfaces with free boundary.
\begin{proposition}\label{prop:freebndry}
Let $\*r:\Sigma\mapsto M\subset \mathbb{R}^3$, $\*r(\partial\Sigma)\subset \Omega$ be a smooth isometric immersion of a surface with boundary, and let $\*v$ be a unit normal to $\Omega$ compatible with the normal $\*n$ to $\Sigma$ (i.e. such that $\langle \*v,\*n\rangle \geq 0)$. Then, $\Sigma$ is a critical point of $\mathcal{W}$ under the free-boundary condition if and only if the following hold:
\begin{equation}\label{eq:freebdry}
\begin{cases}
0 &= \Delta F_H+ H\Delta F_K-\left\langle{h, \He F_K}\right\rangle+F_H|h|^2+HKF_K-HF \text{~~in $\Sigma$},\\
0 &= F_H + \kappa_nF_K \text{~~ on $\partial \Sigma$},\\
0 &=\left\langle{\*v, F \*n-\Big(\nabla_\*T(\tau_g F_K)+h(\nabla F_K,\con)-\nabla_\con F_H-H\nabla_\con F_K\Big)\con}\right\rangle \text{ on $\partial \Sigma$}.
\end{cases}
\end{equation}
\end{proposition}

\begin{proof}
The free-boundary condition implies that any velocity vector field $\*X$ along $\partial \Omega$ must satisfy
\[ \left\langle{\*X, \*v}\right\rangle =0.\]
First, choose $\*X\equiv \*0$ along $\partial \Omega$. The first two equations are established through a similar argument as in the proof of Proposition \ref{prop:fixedboundary}. For the final equation, note that by Theorem \ref{thm:firstvar},
\begin{align*}
0=\delta_\*X\mathcal{W}&=\int_{\partial \Sigma} -F_K h(\con,\*T)\nabla_\*T\left\langle{\*X,\*n}\right\rangle\, ds + \int_{\partial \Sigma} F\left\langle{\*X, \con}\right\rangle\, ds\\
&+\int_{\partial \Sigma} \left\langle{\*X, \*n}\right\rangle\Big(h(\nabla F_K, \con)-\nabla_\con F_H-H\nabla_\con F_K\Big)\, ds.
\end{align*}
Integrating the first term by parts, it follows that
\begin{align*}
0 &= \int_{\partial \Sigma} \left\langle{\*X, \*V}\right\rangle\, ds,\\
\*V &= F\con+\Big(\nabla_\*T(F_K h(\con,\*T))+h(\nabla F_K,\con)-\nabla_\con F_H-H\nabla_\con F_K\Big)\*n.
\end{align*}
Criticality implies that the equations above hold for all $\*X$ such that $\left\langle{\*X, \*v}\right\rangle=0$ along $\partial \Sigma$. Therefore, $\*V$ must be parallel to $\*v$. 

Moreover, observe that $\*v$, $\*n$, and $\con$ lie in the same plane perpendicular to $\partial \Sigma$.  Therefore, the condition that $\*V$ be parallel to $\*v$ translates equivalently to $\langle \*V, R(\*v)\rangle = 0$, where $R$ denotes a 90-degree rotation in this plane.  Consequently, 
\[\left\langle{\*v, F\*n-\Big(\nabla_\*T(F_K h(\con,\*T))+h(\nabla F_K,\con)-\nabla_\con F_H-H\nabla_\con F_K\Big)\con}\right\rangle=0.\]
The last equation now follows under the observation that $h(\*T,\con) = \tau_g$ on $\partial\Sigma$.
\end{proof}

As an immediate consequence of this calculation, we make the following definition. 
\begin{definition}
	$\Sigma$ is called a $\mathcal{W}$-surface with free boundary provided it is a critical point of the functional $\mathcal{W}$ under the free-boundary condition. 
\end{definition}

Notice that when the conformal Willmore functional is considered, meaning when $F=H^2-4K$, these computations recover the critical conditions observed by B. Palmer in \cite{palmer2000} for conformal Willmore surfaces with free boundary.
\begin{equation*}
\begin{cases}
0 &= 2\Delta H+ H(|h|^2-2K) \text{ inside } \Sigma,\\
0 &= h(\con, \con)- h(\*T, \*T)  \text{ on } \partial \Sigma,\\
0 &= (H^2-4K)\left\langle{\*n,\*v}\right\rangle+(2\nabla_{\con}H+4\nabla_\*T h(\*T,\con))\left\langle{\con,\*v}\right\rangle \text{ on } \partial \Sigma.
\end{cases}
\end{equation*}
In particular, note that the second condition implies that the principal curvatures are everywhere equal at the boundary, meaning the boundary must be totally  umbilical. 

%% file: 07-RotSym.tex
\section{Rotational Symmetry}
Free boundary $\mathcal{W}$-surfaces with rotational symmetry will now be studied, leading to the proof of Theorem~\ref{thm:main1}. To that end, suppose that $\Sigma$ has an axis of rotational symmetry, and (without loss of generality) assume $\Sigma$ is symmetric about the $x$-axis. Then, it is possible to truncate $\Sigma$ by some planes perpendicular to the $x$-axis, so that the boundary $\partial\Sigma$ of the truncated surface has components $\partial_i\Sigma$, each of which is circular.

\begin{figure}[htb!]
    \centering
    \def\svgwidth{0.5\linewidth}
    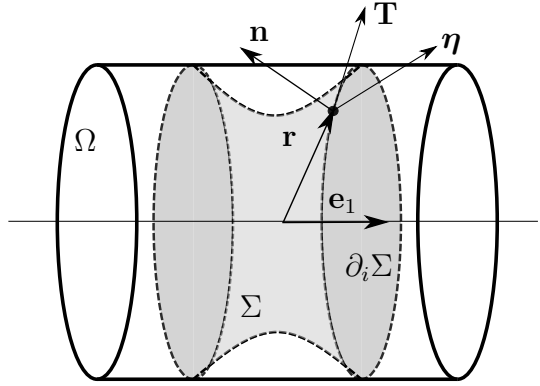
    \caption{A visual aid depicting the surface $\Sigma$ with boundary components $\partial_i\Sigma$ and support surface $\Omega$, along with the frame $\{\*T,\*n,\con\}$, position vector $\*r$, and axis of symmetry $\*e_1$.}
    \label{fig:genlsetup}
\end{figure}

Let $\*T$ denote a choice of a unit tangential vector field along $\partial \Sigma$. At each point along this boundary, it follows that the position vector $\*r$, the co-normal vector $\con$, the normal vector $\*n$, and the constant vector $\*e=\*e_1$ are all co-planar, since all are perpendicular to $\*T$ (see Figure~\ref{fig:genlsetup} for an illustration).  Moreover, note that along $\partial\Sigma$ we also have $\tau_g = h(\*T,\con) = 0$, $\nabla_\con\langle \*e,\*n\rangle = -h(\con,\con)\langle\*e,\con\rangle$, and $h(\nabla f,\con) = h(\con,\con)\nabla_\con f$ for any smooth $f:\Sigma \to \mathbb{R}$.  Thus, the flux formula (\ref{eq:transflux}) becomes
\begin{align*}
& 0 = \int_{\partial \Sigma} (F_H+h(\*T,\*T)F_K)\nabla_\con \left\langle{\*e, \*n}\right\rangle \,ds \\
&\hspace{-0.5pc}- \int_{\partial \Sigma} \left\langle{\*e, \*n}\right\rangle(\nabla_\con F_H+h(\*T,\*T)\nabla_\con F_K)\,ds + \int_{\partial \Sigma}F\left\langle{\*e, \con}\right\rangle \,ds\\
&\hspace{-0.5pc}= \int_{\partial \Sigma} \left\langle{\*e,\Big(F-h(\con,\con)(F_H+h(\*T,\*T)F_K)\Big)\con -\Big(\nabla_\con F_H+h(\*T,\*T)\nabla_\con F_K\Big)\*n}\right\rangle\,ds.\\
\end{align*}
Similarly, (\ref{eq:scaleflux}) becomes
\begin{align*}
& \int_{\Sigma}(2F-HF_H-2KF_K)\,d\mu = \int_{\partial \Sigma} (F_H+h(\*T,\*T)F_K)\nabla_\con \left\langle{\*r,\*n}\right\rangle\, ds \\
&\hspace{-0.5pc}- \int_{\partial \Sigma} \left\langle{\*r,\*n}\right\rangle(\nabla_\con F_H+h(\*T,\*T)\nabla_\con F_K) \,ds + \int_{\partial \Sigma}F\left\langle{\*r, \con}\right\rangle\, ds\\
&\hspace{-0.5pc}= \int_{\partial \Sigma} \left\langle{\*r,\Big(F-h(\con,\con)(F_H+h(\*T,\*T)F_K)\Big)\con -\Big(\nabla_\con F_H+h(\*T,\*T)\nabla_\con F_K\Big)\*n}\right\rangle\,ds.
\end{align*}
Therefore, let
\begin{align*}
\*V&=\Big(F-h(\con,\con)(F_H+h(\*T,\*T)F_K)\Big)\con -\Big(\nabla_\con F_H+h(\*T,\*T)\nabla_\con F_K \Big)\*n ,\\
\ell_i &= |\partial_i \Sigma|.
\end{align*} 
Then, since the principal curvatures are constant along each component of the boundary, the calculations above imply
\begin{align}
\label{eq:rotsymflux1}
0 & = \sum_i \ell_i \left\langle{\*e, \*V|_{\partial_i \Sigma}}\right\rangle;\\
\label{eq:rotsymflux2}
\int_{\Sigma}(2F-HF_H-2KF_K)\,d\mu & = \sum_i \ell_i \left\langle{\*r|_{\partial_i \Sigma}, \*V|_{\partial_i \Sigma}}\right\rangle.
\end{align}
\begin{remark}
The above formulas are comparable with those in \cite[Proof of Theorem 3]{dall2013}. In that paper, $\Sigma$ satisfies an additional reflection symmetry. As a consequence, it is possible to truncate $\Sigma$ such that, for some $i$,  $\left\langle{\*r|_{\partial_i \Sigma}, \*V|_{\partial_i \Sigma}}\right\rangle=0$. Also, that article and \cite{vassilev2014,dall2008,deckelnick2009} show there are plenty of Willmore surfaces with rotational symmetry. 
\end{remark}

For the rest of this section, it is assumed that $\mathcal{W}$ is scaling invariant. In this case, recall that 
\[2F-HF_H-2KF_K=0.\]
Since $\tau_g = h(\*T,\con) \equiv 0$ on $\partial\Sigma$ (the boundary is a line of curvature), it follows from this invariance that $\*V|_{\partial_i \Sigma}$ reduces to

\begin{equation}\label{eq:Vonbndry}
    \*V|_{\partial_i \Sigma}=\frac{1}{2}\Big(\kappa_n-h(\con, \con)\Big)F_H\con -\Big(\nabla_\con F_H+\kappa_n\nabla_\con F_K \Big)\*n.
\end{equation}

The following result will be used repeatedly in the proof of our main theorems. 
\begin{lemma}\label{lem:Vis0}
Let $\Sigma$ be a rotationally symmetric $\mathcal{W}$-surface and suppose $\mathcal{W}$ is scaling invariant. The following are equivalent.
\begin{enumerate}
	\item  $\*V = \*0$ on at least one boundary component $\partial_i\Sigma$.
	\item $\*V\equiv \*0$ on $\partial\Sigma$.
	\item Either $\Sigma$ is spherical or  $F_H\equiv0$, $F_K=c$ for some constant c, and $F=cK$ on $\Sigma$.
\end{enumerate}
\end{lemma}

\begin{proof}
First, we show that $(1)\rightarrow (2)$. Without loss of generality, we can assume that $\Sigma$ has at most 2 boundary components. If $\partial \Sigma$ has one component, then the statement follows vacuously. If $\partial \Sigma$ has two connected components, assume that $\*V_j\neq \*0$ for $j\neq i$. Then, we can choose the origin along the $x$-axis, in which case $\left\langle{\*r|_{\partial_j \Sigma}, \*V|_{\partial_j \Sigma}}\right\rangle\neq 0$ for $j\neq i$. But this contradicts equation (\ref{eq:rotsymflux2}),
\[0 = \sum_k \ell_k \left\langle{\*r|_{\partial_k \Sigma}, \*V|_{\partial_k \Sigma}}\right\rangle, \] so we must have $\*V\equiv\*0$ on $\partial\Sigma$ in this case as well.

Next, we'll show $(2)\rightarrow (3)$. Observe that (\ref{eq:rotsymflux1}) and (\ref{eq:rotsymflux2}) hold for any truncated surface. As a consequence, the above argument can be repeated for a sequence of surfaces whose boundaries exhaust $\Sigma$ to obtain that 
$\*V\equiv \*0$. Thus,  on the entire surface $\Sigma$,
\begin{align*}
0 &= (h(\*T,\*T)-h(\con, \con))F_H,\\
0 &= \nabla_\con F_H+h(\*T,\*T)\nabla_\con F_K, 
\end{align*}
where $\con = \*T\times\*n$ is the unique extension of the outward co-normal field to the interior of $\Sigma$ (which exists since $\Sigma$ is rotational).
It follows by continuity that for any connected and rotationally symmetric submanifold $\Sigma_0\subset\Sigma$, either $F_H=0$ or $h(\*T,\*T)=h(\con,\con)$ holds.  We consider two possible cases:

\textit{Case 1:} $h(\*T,\*T)=h(\con,\con)$ and $\Sigma_0$ is totally umbilical.  By rotational symmetry, $\Sigma_0\subset\Sigma$ must be spherical, hence both $H$ and $K$ are constant on $\Sigma_0$. Moreover, this implies that $F_H$ is constant on $\Sigma_0$ as well.  On the other hand, if $\Sigma$ is not entirely spherical, then taking $\Sigma_0$ to be the maximally connected spherical submanifold contained in $\Sigma$ (guaranteed by Zorn's Lemma) we conclude there is a rotationally symmetric $\Sigma_1\subset\Sigma$ which is nonspherical and has nontrivial intersection with $\Sigma_0$.  
This is only possible if $F_H\equiv 0$ everywhere on $\Sigma$.



\textit{Case 2:} $F_H\equiv 0$ on $\Sigma$. Then, the second equation implies that 
\[h(\*T,\*T)\nabla_\con F_K=0.\]
Again, it must hold for any connected and rotationally symmetric $\Sigma_0\subset\Sigma$ that either $h(\*T,\*T)=0$ or $\nabla_\con F_K=0$.  Suppose $h(\*T,\*T)=0$ on $\Sigma_0$. By rotational symmetry, $\nabla_\*T \*T$ is perpendicular to both $\*e, \*T$. Thus, $h(\*T,\*T)=0$ if and only if $\*n\parallel \*e$. As a consequence, $\*n=\pm \*e$, so $h(\con, \con)=0$ and $\Sigma_0$ is flat.  Hence, $F_K$ is constant on $\Sigma_0$.  On the other hand, if $\nabla_\con F_K=0$ on $\Sigma_0$, then rotational symmetry implies that $F_K$ is constant on $\Sigma_0$ also since $\nabla_\*T F_K = 0$.  Therefore, in either case it follows that $F_H\equiv 0$ and $F_K$ is constant, implying $F=cK$ on $\Sigma_0$.  Taking an overlapping sequence of rotationally symmetric subsurfaces which exhaust $\Sigma$ then yields the conclusion.

Finally, consider the implication $(3)\rightarrow (1)$.  This is clear from the expression of  $\*V|_{\partial_i\Sigma}$ in (\ref{eq:Vonbndry}) when $F_H\equiv 0, F_K=c\in\mathbb{R}$.  Also, $\Sigma$ spherical implies $\kappa_n = h(\con,\con)$ as well as $H,K$ constant, so that the implication again follows from (\ref{eq:Vonbndry}).
\end{proof}

\begin{lemma}\label{lem:neworigin}
Let $\mathcal{W}$ be scaling invariant and let $\Sigma$ be a rotationally symmetric $\mathcal{W}$-surface.  Denote the unit vector aligned with the axis of rotation by $\*e$.  Either $\left\langle{\*e, \*V}\right\rangle=0$ holds on $\Sigma$ or it is possible to choose an origin such that $\left\langle{\*r,\*V}\right\rangle=0$ on $\Sigma$.
\end{lemma} 
\begin{proof}
Suppose there is some boundary component $\partial_i \Sigma$ on which $\left\langle{\*e, \*V}\right\rangle\neq 0$. Since $\*V$ is a constant combination of $\con,\*n$ on $\partial_i\Sigma$, it is then possible using planar geometry (see e.g. Figure~\ref{fig:neworigin}) to choose an origin on the $x$-axis such that $\left\langle{\*r,\*V}\right\rangle=0$ there.  More precisely, because $\partial_i\Sigma$ is circular, the planes spanned by $\*T$ and $\*T\times\*V$ at each point will intersect the $x$-axis in a common point (when $\*V \nparallel \con$), which serves as the new origin of $\*r$ (when $\*V \parallel \con$ the new origin can be placed at the center of $\partial_i\Sigma$).  Applying the flux formula corresponding to dilation now leads to $\left\langle{\*r,\*V}\right\rangle=0$ on $\Sigma$.   
\end{proof}

\begin{figure}[htb!]
    \centering
    \def\svgwidth{0.5\linewidth}
    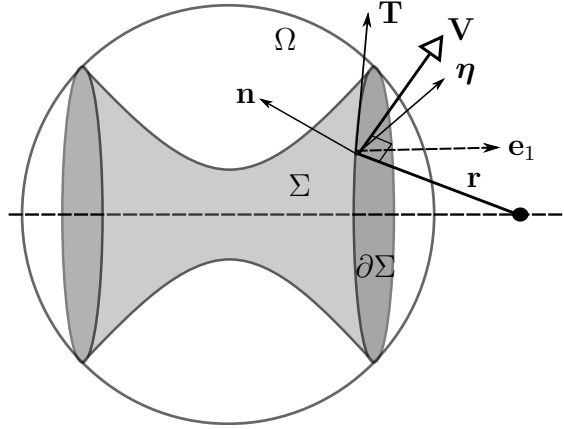
    \caption{An illustration depicting the new choice of origin in Lemma~\ref{lem:neworigin}.  Note that $\*r$ is orthogonal to $\*V$, as desired.}
    \label{fig:neworigin}
\end{figure}

Next, we consider the surfaces with free boundary in addition to rotational symmetry. That is, suppose $\Sigma, \Omega \subset \mathbb{R}^3$ share a common axis of rotational symmetry and satisfy the system of equations (\ref{eq:freebdry}). Again, without loss of generality, it may be assumed that the axis of symmetry is the $x$-axis, and $\Sigma$ has at most 2 boundary components. Furthermore, due to the shared symmetry, each connected component $\partial_i \Sigma$ of $\partial \Sigma$ is circular around the $x$-axis. Moreover, it follows that $\Sigma$ intersects $\Omega$ at a constant angle. 
Consequentially, $\left\langle{\*n, \*v}\right\rangle$ is constant along $\partial \Sigma$, where $\*v$ is an appropriate normal vector to $\Omega$.  This leads to the following well known observation illustrated in Figure~\ref{fig:twoplanes}.
\begin{lemma}
	\label{lem:curvatureline}
	Suppose that $\Sigma$ is a surface which has free boundary with respect to a support surface $\Omega$.  If $\Sigma$ intersects $\Omega$ at a constant nonzero angle, then $\partial \Sigma\subset \Omega$ is formed by lines of curvature if and only if so is $\partial \Sigma \subset \Sigma$.  
\end{lemma}
\begin{proof}
	Let $\*T$ be a unit vector field tangent to $\partial \Sigma$. Since $\left\langle{\*n, \*v}\right\rangle$ is constant, we have
	\begin{align*}
	0 &= \nabla_\*T \left\langle{\*n, \*v}\right\rangle\\
	&= \left\langle{\nabla_\*T \*n, \*v}\right\rangle+\left\langle{\*n, \nabla_\*T \*v}\right\rangle\\
	&= h^{\Sigma}(\*T,\con)\left\langle{\con, \*v}\right\rangle+h^{\Omega}(\bm{\zeta}, \*T)\left\langle{\*n, \bm{\zeta}}\right\rangle, 
	\end{align*}
	where $\bm{\zeta}$ is an appropriate unit co-normal vector to $\partial \Sigma\subset \Omega$. Since $\{\*n, \con\}$ and $\{\*v, \bm{\zeta}\}$ are pairs of perpendicular vectors in the same plane and the angle between $\*n$ and $\*v$ is nonzero, the result follows. 
\end{proof}

\begin{figure}[htb!]
    \centering
    \def\svgwidth{0.5\linewidth}
    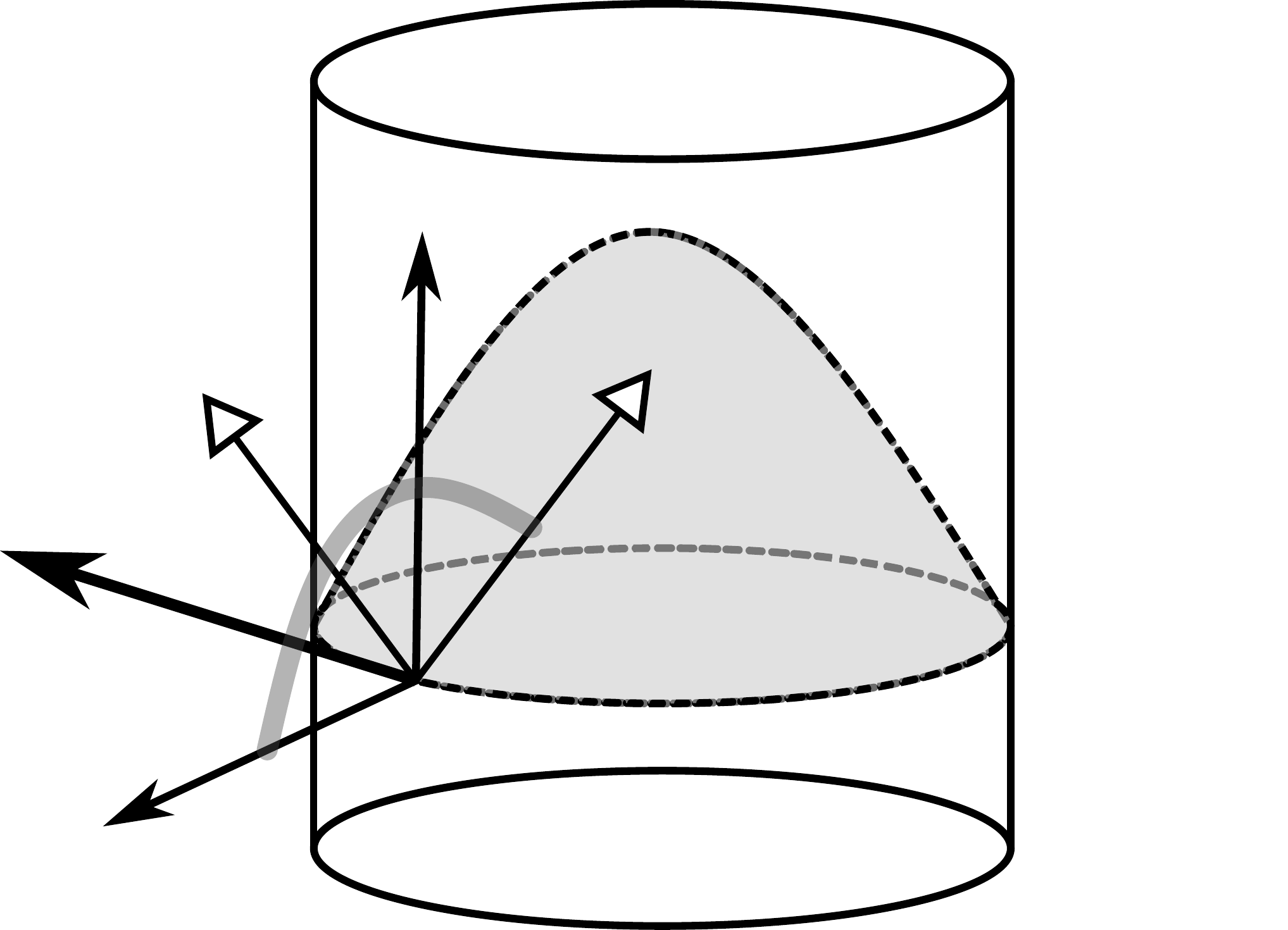
    \caption{An illustration depicting  Lemma~\ref{lem:curvatureline}. Note that $\*v,\*n,\bm{\zeta},\con$ are all co-planar.}
    \label{fig:twoplanes}
\end{figure}

As a consequence of this, $\partial_i \Sigma$ is a line of curvature on both $\Sigma$ and $\Omega$. By the calculations above Lemma~\ref{lem:Vis0}, $\Sigma$ is a rotationally symmetric and scale-invariant $\mathcal{W}$-surface with free boundary with respect to $\Omega$ if and only if the following hold:
\begin{align}
\*V &= \frac{1}{2}\Big(h(\*T,\*T)-h(\con, \con)\Big)F_H\con -\Big(\nabla_\con F_H+h(\*T,\*T)\nabla_\con F_K \Big)\*n,\\
\label{eq:rotflux1}
0 & = \sum_i \ell_i \left\langle{\*e, \*V_i}\right\rangle,\\
\label{eq:rotflux2}
0 & = \sum_i \ell_i \left\langle{\*r_i, \*V_i}\right\rangle,\\
0 &= \Delta F_H+ H\Delta F_K-\left\langle{h, \He F_K}\right\rangle+F_H|h|^2+HKF_K-HF, \text{~~on $\Sigma$},\\
\label{eq:rotfree1}
0 &= F_H+\kappa_nF_K \text{~~ on $\partial \Sigma$},\\
\label{eq:rotfree2}
\*v & \parallel \*V \,\,\mathrm{or}\,\, \*V=\*0 \text{~~ on $\partial \Sigma$}.
\end{align}
Note that the last equation follows from arguments similar to those in the proof of  Proposition~\ref{prop:freebndry}.  An immediate consequence of this system is the following.
\begin{proposition}\label{prop:onebdry}
Suppose that $\mathcal{W}$ is scaling invariant and $\Sigma$ is a rotationally symmetric free-boundary $\mathcal{W}$-surface with exactly one boundary component.  Then, one of the following holds:
\begin{enumerate}
\item $\Sigma$ is spherical and $F\equiv 0$.
\item $F_H\equiv 0$, $F_K$ is constant. 
\end{enumerate}
\end{proposition}

\begin{proof}
Since there is only one boundary component, by (\ref{eq:rotflux1}) and (\ref{eq:rotflux2}), we have
\[ 0=\left\langle{\*e, \*V}\right\rangle=\left\langle{\*r, \*V}\right\rangle.\]
Moreover, $\*e$ and $\*r$ are not parallel, so $\*V\equiv 0$. By Lemma \ref{lem:Vis0}, either $\Sigma$ is spherical or $F_H\equiv 0$ and $F_K$ is constant. In the former case, note that $K = H^2/4 = \kappa_n^2$ on $\partial\Sigma$, so that multiplying equation (\ref{eq:rotfree1}) by $\kappa_n$ and using scaling invariance yields that $F=0$ on $\partial \Sigma$.  Moreover, since $\Sigma$ is spherical, $H$ and $K$ are constant and hence $F$ is constantly equal to zero on $\Sigma$.
\end{proof}


It is now possible to eliminate the dependence on the number of boundary components, hence establishing Theorem~\ref{thm:main1}.

\begin{proof}[Proof of Theorem~\ref{thm:main1}]
Without loss of generality, suppose that $\Sigma$ has at most 2 boundary components. If $\Sigma$ has only one component, then the result follows from Proposition \ref{prop:onebdry}. So, suppose that $\Sigma$ has 2 boundary components. 

Let $\Pi$ be any plane containing the $x$-axis. Since $\Omega$ is strictly convex, its intersection with $\Pi$ is a convex simple curve $\gamma$. It follows that the normal vector map of $\gamma$ in $\Pi$, from $\gamma$ to the unit circle, is one-to-one and onto \cite[Chapter~6]{gray2006}. Furthermore, each boundary component of $\partial\Sigma$ intersects $\Pi$ at two points whose normal vectors are symmetrical over the $x$-axis. Since $\Sigma$ has 2 boundary components, its intersection with $\Pi$ consists of 4 points, 2 on each side of the $x$-axis. Now, apply Lemma~\ref{lem:neworigin} to choose the origin on the $x$-axis such that $\left\langle{\*r|_{\partial_i \Sigma}, \*V|_{\partial_i \Sigma}}\right\rangle=0$ for some $i$.  If $\*V|_{\partial_i \Sigma}=\*0$, then by Lemma \ref{lem:Vis0} either $\Sigma$ is spherical or $F_H\equiv 0$ and $F_K$ is constant.  Therefore, the conclusion follows from an argument similar to the proof of Proposition~\ref{prop:onebdry}.  

Conversely, suppose $\*V|_{\partial_i \Sigma}\neq 0$.  Since $\*v \parallel \*V$ on $\partial_i\Sigma$ from \eqref{eq:rotfree2}, it follows that $\left\langle{\*r|_{\partial_i \Sigma}, \*v|_{\partial_i \Sigma}}\right\rangle=0$. Moreover, $\gamma$ is strictly convex, therefore the strictly convex region it bounds lies entirely on one side of the tangent line directed with $\*r|_{\partial_i \Sigma}$.  As a consequence, $\left\langle{\*r|_{\partial_j \Sigma}, \*v|_{\partial_j \Sigma}}\right\rangle\neq 0$ for $j\neq i$. By the flux formula (\ref{eq:rotflux2}) and the free boundary condition (\ref{eq:rotfree2}), it follows that $\*v$ cannot be parallel to $\*V$ along $\partial_j\Sigma$, so that
\[\*V|_{\partial_j \Sigma}=\*0.\]
The rest now follows from the preceding argument.
\end{proof}
\begin{remark}\label{example1} The first case could happen if, for example, $F=(H^2-4K)\frac{K}{H^2}$. The second case can happen if, for example, $F=4K$ and $\Sigma$ is a $C^2$ surface which is rotationally symmetric and flat in a neighborhood of its boundary.  
\end{remark}


With this in place, Theorem~\ref{thm:main1} will now be used to establish some interesting results about the free-boundary critical points of the conformal Willmore functional.  Such surfaces $\Sigma\subset\mathbb{R}^3$ are known as conformal Willmore surfaces.

\begin{corollary}
Let  $\Sigma\subset \mathbb{R}^3$ be an immersed conformal Willmore surface which has free boundary with respect to $\Omega$. Suppose that $\Sigma$ and $\Omega$ share a common axis of rotational symmetry, and $\Omega$ is strictly convex. Then $\Sigma$ must be either spherical or flat. 
\end{corollary}

\begin{proof}
For a conformal Willmore surface, $F=H^2-4K$ so $F_H=2H$ and $F_K=4$. Thus, by Theorem \ref{thm:main1}, either $\Sigma$ is spherical or $F_H=2H\equiv 0$. Thus, $\Sigma$ is minimal. Equation \ref{eq:rotfree1} implies $\kappa_n=0$ on $\partial \Omega$. The conclusion follows from the well known fact that rotationally symmetric minimal surfaces are flat.  Indeed, such surfaces satisfy $h(\con,\con) = -h(\*T,\*T) = c\in\mathbb{R}$, so that $K = h(\*T,\con)^2 = 0$ since $h(\*T,\con) = 0$ on any circular curve perpendicular to the axis of rotation.
\end{proof}

Indeed, as mentioned in the Introduction, the convexity assumption is actually unnecessary in this case.
\begin{theorem}\label{thm:noconvex}
Let  $\Sigma \subset \mathbb{R}^3$ be an immersed conformal Willmore surface that has free boundary with respect to $\Omega$. Suppose that $\Sigma$ and $\Omega$ share a common axis of rotational symmetry, and $\Sigma$ intersects $\Omega$ transversally. Then $\Sigma$ must be either spherical or flat. 
\end{theorem} 
\begin{proof}
Since $F=H^2-4K$, by equation (\ref{eq:rotfree1}),
\begin{equation}\label{eq:umb}
    h(\*T,\*T)=h(\con, \con) \quad \text{on}\,\,\partial\Sigma.
\end{equation}
Thus, it follows from (\ref{eq:Vonbndry}) that
\[\*V|_{\partial\Sigma}=-(\nabla_\con H)\*n.\]
Since $\Sigma$ meets $\Omega$ transversally, $\*n$ is not parallel to $\*v$ on $\partial \Sigma$. Therefore, by equation (\ref{eq:rotfree2}), $\*V=\*0$ on $\partial \Sigma$. By Lemma~\ref{lem:Vis0}, $\*V=\*0$ on $\Sigma$ and  $\Sigma$ is either spherical or minimal. When $\Sigma$ is not spherical, it is  minimal and rotationally symmetric, hence must be flat.
\end{proof}


\begin{remark}
Note that the totally umbilical condition (\ref{eq:umb}) is equivalent to the condition in \cite{palmer2000} that $H = 2 \kappa_n$, since $\kappa_n = H - h(\con,\con)$ on $\partial\Sigma$.
\end{remark}

%% file: figs/catenoid.pdf_tex
\begingroup%
  \makeatletter%
  \providecommand\color[2][]{%
    \errmessage{(Inkscape) Color is used for the text in Inkscape, but the package 'color.sty' is not loaded}%
    \renewcommand\color[2][]{}%
  }%
  \providecommand\transparent[1]{%
    \errmessage{(Inkscape) Transparency is used (non-zero) for the text in Inkscape, but the package 'transparent.sty' is not loaded}%
    \renewcommand\transparent[1]{}%
  }%
  \providecommand\rotatebox[2]{#2}%
  \ifx\svgwidth\undefined%
    \setlength{\unitlength}{600bp}%
    \ifx\svgscale\undefined%
      \relax%
    \else%
      \setlength{\unitlength}{\unitlength * \real{\svgscale}}%
    \fi%
  \else%
    \setlength{\unitlength}{\svgwidth}%
  \fi%
  \global\let\svgwidth\undefined%
  \global\let\svgscale\undefined%
  \makeatother%
  \begin{picture}(1,0.75)%
    \put(0,0){\includegraphics[width=\unitlength,page=1]{catenoid.pdf}}%
    \put(0.23242188,0.68359375){\color[rgb]{0,0,0}\makebox(0,0)[lb]{\smash{}}}%
    \put(0.52112724,0.49561115){\color[rgb]{0,0,0}\makebox(0,0)[lb]{\smash{$\*r$}}}%
    \put(0.46343098,0.67932942){\color[rgb]{0,0,0}\makebox(0,0)[lb]{\smash{$\*n$}}}%
    \put(0.68022781,0.7086263){\color[rgb]{0,0,0}\makebox(0,0)[lb]{\smash{$\*T$}}}%
    \put(0.60210281,0.38636066){\color[rgb]{0,0,0}\makebox(0,0)[lb]{\smash{$\*e_1$}}}%
    \put(0.79882812,0.66601562){\color[rgb]{0,0,0}\makebox(0,0)[lb]{\smash{$\con$}}}%
    \put(0.15820312,0.48632812){\color[rgb]{0,0,0}\makebox(0,0)[lb]{\smash{$\Omega$}}}%
    \put(0.44726562,0.19726562){\color[rgb]{0,0,0}\makebox(0,0)[lb]{\smash{$\Sigma$}}}%
    \put(0.63085938,0.26757812){\color[rgb]{0,0,0}\makebox(0,0)[lb]{\smash{$\partial_i\Sigma$}}}%
  \end{picture}%
\endgroup%

%% file: figs/neworigin.pdf_tex
\begingroup%
  \makeatletter%
  \providecommand\color[2][]{%
    \errmessage{(Inkscape) Color is used for the text in Inkscape, but the package 'color.sty' is not loaded}%
    \renewcommand\color[2][]{}%
  }%
  \providecommand\transparent[1]{%
    \errmessage{(Inkscape) Transparency is used (non-zero) for the text in Inkscape, but the package 'transparent.sty' is not loaded}%
    \renewcommand\transparent[1]{}%
  }%
  \providecommand\rotatebox[2]{#2}%
  \ifx\svgwidth\undefined%
    \setlength{\unitlength}{600bp}%
    \ifx\svgscale\undefined%
      \relax%
    \else%
      \setlength{\unitlength}{\unitlength * \real{\svgscale}}%
    \fi%
  \else%
    \setlength{\unitlength}{\svgwidth}%
  \fi%
  \global\let\svgwidth\undefined%
  \global\let\svgscale\undefined%
  \makeatother%
  \begin{picture}(1,0.75)%
    \put(0,0){\includegraphics[width=\unitlength,page=1]{neworigin.pdf}}%
    \put(0.82083961,0.42051491){\color[rgb]{0,0,0}\makebox(0,0)[lb]{\smash{$\*r$}}}%
    \put(0.89257812,0.4765625){\color[rgb]{0,0,0}\makebox(0,0)[lb]{\smash{$\*e_1$}}}%
    \put(0.78710938,0.67773438){\color[rgb]{0,0,0}\makebox(0,0)[lb]{\smash{$\*V$}}}%
    \put(0.80078125,0.61132812){\color[rgb]{0,0,0}\makebox(0,0)[lb]{\smash{$\con$}}}%
    \put(0.66601562,0.70703125){\color[rgb]{0,0,0}\makebox(0,0)[lb]{\smash{$\*T$}}}%
    \put(0.41795315,0.56746689){\color[rgb]{0,0,0}\makebox(0,0)[lb]{\smash{$\*n$}}}%
    \put(0.484375,0.66015625){\color[rgb]{0,0,0}\makebox(0,0)[lb]{\smash{$\Omega$}}}%
    \put(0.50976562,0.40625){\color[rgb]{0,0,0}\makebox(0,0)[lb]{\smash{$\Sigma$}}}%
    \put(0.62664196,0.2639938){\color[rgb]{0,0,0}\makebox(0,0)[lb]{\smash{$\partial\Sigma$}}}%
  \end{picture}%
\endgroup%

%% file: figs/twoplanes.pdf_tex
\begingroup%
  \makeatletter%
  \providecommand\color[2][]{%
    \errmessage{(Inkscape) Color is used for the text in Inkscape, but the package 'color.sty' is not loaded}%
    \renewcommand\color[2][]{}%
  }%
  \providecommand\transparent[1]{%
    \errmessage{(Inkscape) Transparency is used (non-zero) for the text in Inkscape, but the package 'transparent.sty' is not loaded}%
    \renewcommand\transparent[1]{}%
  }%
  \providecommand\rotatebox[2]{#2}%
  \ifx\svgwidth\undefined%
    \setlength{\unitlength}{585.88837871bp}%
    \ifx\svgscale\undefined%
      \relax%
    \else%
      \setlength{\unitlength}{\unitlength * \real{\svgscale}}%
    \fi%
  \else%
    \setlength{\unitlength}{\svgwidth}%
  \fi%
  \global\let\svgwidth\undefined%
  \global\let\svgscale\undefined%
  \makeatother%
  \begin{picture}(1,0.72248869)%
    \put(0,0){\includegraphics[width=\unitlength,page=1]{twoplanes.pdf}}%
    \put(0.03285048,0.31797452){\color[rgb]{0,0,0}\rotatebox{-15.242082}{\makebox(0,0)[lb]{\smash{$\*T$}}}}%
    \put(0.14344618,0.43355187){\color[rgb]{0,0,0}\rotatebox{-5.005786}{\makebox(0,0)[lb]{\smash{$\*n$}}}}%
    \put(0.07424479,0.11727564){\color[rgb]{0,0,0}\rotatebox{18.665666}{\makebox(0,0)[lb]{\smash{$\*v$}}}}%
    \put(0.27246691,0.49818342){\color[rgb]{0,0,0}\rotatebox{0.04522472}{\makebox(0,0)[lb]{\smash{$\bm{\zeta}$}}}}%
    \put(0.46384627,0.45163169){\color[rgb]{0,0,0}\rotatebox{0.04522472}{\makebox(0,0)[lb]{\smash{$\con$}}}}%
    \put(0.3536509,0.65522238){\color[rgb]{0,0,0}\makebox(0,0)[lb]{\smash{$\Omega$}}}%
    \put(0.5790397,0.2090137){\color[rgb]{0,0,0}\makebox(0,0)[lb]{\smash{$\partial\Sigma$}}}%
    \put(0.59055804,0.35858311){\color[rgb]{0,0,0}\makebox(0,0)[lb]{\smash{$\Sigma$}}}%
  \end{picture}%
\endgroup%

%% file: 08-NoScaleInvar.tex
\section{Functionals Without Scaling Invariance}
Many important functionals do not share the dilation-invariance seen in the Willmore energy.  It is easy to verify that even very similar functionals such as the Helfrich-Canham energy (\ref{eq:helfrich}) do not remain static when a surface is rescaled.  Because of this, it is enlightening to also examine the properties of $\mathcal{W}$-functionals that are not scaling invariant.  In particular, flux formula (\ref{eq:scaleflux}) can be used to show the following.

\begin{lemma}\label{lem:bndryctrl}
The equation
\begin{equation*}
    \int_\Sigma \left(2F - HF_H - 2KF_K\right)d\mu = 0,
\end{equation*}
holds for any $\mathcal{W}$-critical surface immersion $\*r(\Sigma)$ provided the following expressions hold on $\partial\Sigma$:
\begin{align}
    0 &= \tau_g F_H, \label{eq:bcond1} \\
    0 &= F - h(\con,\con)F_H -KF_K, \label{eq:bcond2} \\
    0 &= h(\nabla F_K,\con) - \nabla_\con F_H - H\nabla_\con F_K. \label{eq:bcond3}
\end{align}
\end{lemma}

\begin{proof}
First, notice that $\nabla_\con \*r \perp \*n$, so that
\begin{equation*}
    \nabla_\con\langle \*r, \*n\rangle = \langle \*r, \nabla_\con \*n\rangle = -\left\langle \*r, \tau_g \*T + h(\con,\con)\con\right\rangle.
\end{equation*}
Also, writing $\nabla \langle \*r, \*n\rangle = \nabla_\*T \langle \*r, \*n\rangle \*T + \nabla_\con \langle \*r, \*n\rangle\con$, it follows that
\begin{equation*}
    h\left(\nabla\langle \*r, \*n\rangle, \con\right) = -\left\langle \*r, H\tau_g \*T + \left( \tau_g^2 + h(\con,\con)^2\right)\con \right\rangle.
\end{equation*}
With this, the right-hand side of equation  (\ref{eq:scaleflux}) becomes
\begin{align*}
&\int_{\partial \Sigma} \Big((F_H+HF_K)\nabla_\con \left\langle{\*r, \*n}\right\rangle - F_K h(\nabla \left\langle{\*r, \*n}\right\rangle, \con)\Big)\, ds\\
&+\int_{\partial \Sigma} \left\langle{\*r, \*n}\right\rangle\Big(h(\nabla F_K, \con)-\nabla_\con F_H-H\nabla_\con F_K\Big)\, ds + \int_{\partial \Sigma} F\left\langle{\*r, \con}\right\rangle \,ds \\
&= \int_{\partial\Sigma} \left\langle \*r, -\tau_g F_H \*T + \left(F - (F_H + H F_K)h(\con,\con) + \left(\tau_g^2 + h(\con,\con)^2\right)F_K\right)\con \right\rangle\, ds \\
&+ \int_{\partial\Sigma} \left\langle \*r, \left(h(\nabla F_K, \con)-\nabla_\con F_H-H\nabla_\con F_K\right)\*n \right\rangle \, ds.
\end{align*}
The inner product $\langle \*r,\*V\rangle$ inside the above integral expression vanishes for any immersion $\*r$ when each component of the vector field $\*V$ vanishes identically.  This combined with the fact that $H = h(\*T,\*T) + h(\con,\con)$ on $\partial\Sigma$ now yields the claimed boundary conditions.
\end{proof}

\begin{remark}
For immersions that meet a planar boundary tangentially, $\kappa_n = \tau_g = 0$ and the flux formula reduces to 
\[ \int_\Sigma \left(2F - HF_H - 2KF_K\right) d\mu = \int_{\partial\Sigma} (F-HF_H) \langle \*r,\con\rangle\, ds. \]  In the case where $F$ is the (scaling invariant) Willmore functional, this implies the result of Dall'Acqua \cite{dall2012} that $h(\*X,\*Y) = 0$ for all vector fields $\*X,\*Y$ tangent to $\Sigma$ at the boundary.  More precisely, $\langle \*r, \con \rangle$ is of constant sign on $\partial\Sigma$, so it follows that $H \equiv 0$ there.  Since $\kappa_n$ is zero also, this implies that $h(\con,\con) = 0$, so both principal curvatures must be zero on $\partial\Sigma$. 
\end{remark}

It is now appropriate to give the proof of Theorem~\ref{thm:main2}, which details a situation where conditions on the boundary of a $\mathcal{W}$-critical surface can exert control over the interior.

\begin{proof}[Proof of Theorem~\ref{thm:main2}]
It follows from the hypotheses and Lemma~\ref{lem:bndryctrl} that any critical surface $\Sigma$ must satisfy the integral equality \[\int_\Sigma \left(2F - HF_H - 2KF_K\right) d\mu = 0.\]  However, $\mathcal{W}$ is assumed to be shrinking or expanding, so the integrand is vanishing. 
\end{proof}

\subsection{Corollaries}
With Theorem~\ref{thm:main2} now established, some interesting corollaries can be extracted.  First, consider the case where $\Sigma$ has an axis of rotational symmetry.
\begin{corollary}\label{cor:main2}
Let $\mathcal{W}$ be expanding or shrinking, and let $\Sigma$ be a rotationally symmetric $\mathcal{W}$-surface with boundary. Suppose additionally that the following hold:
\begin{enumerate}
\item $F-h(\con, \con)F_H-KF_K=0$ on $\partial\Sigma$,
\item $\nabla_\eta F_H + \kappa_n \nabla_\eta F_K = 0$ on $\partial\Sigma$.
\end{enumerate}
Then, either $\Sigma$ is spherical or there is a constant $c$ such that $F_H\equiv 0$, $F_K\equiv c$, and $F\equiv cK$ on $\Sigma$.
\end{corollary}


\begin{proof}
Since $\Sigma$ is rotationally symmetric by assumption, its boundaries are lines of curvature.  As such, $\tau_g \equiv 0$ on $\partial\Sigma$, hence (\ref{eq:bcond1}) is satisfied. The result then follows from Theorem \ref{thm:main2} and Lemma \ref{lem:Vis0}. 
\end{proof}
\begin{remark}\label{rem:counterex}
	For example, the first case occurs when $F=(H^2-4K)^2$, while the second case happens if $F=H^4+K$ and $\Sigma$ is minimal.  
\end{remark}

Moreover, it is worthwhile to consider functionals which are independent of the Gauss curvature $K$, as many of these objects appear quite naturally in practice, e.g. the surface area, total mean curvature, and (non-conformal) Willmore functionals.  To that end, there is the following Corollary which details the case where $F = F(H)$ is a real analytic function of $H$ alone. 
\begin{corollary}\label{cor:minimal}
Let $\mathcal{W}$ be expanding or shrinking, $\Sigma$ be a $\mathcal{W}$-critical surface, and $F = F(H)$ be a real analytic function of $H$ alone.  Suppose $F = F_H = \nabla_\con F_H = 0$ on $\partial\Sigma$.  Then, one of the following holds:
\begin{enumerate}
    \item $F \equiv 0$ everywhere on $\Sigma$,
    \item $F \equiv cH^2$ for some $c\in\mathbb{R}$ and $\mathcal{W}$ is scaling invariant,
    \item $\Sigma$ has constant mean curvature and $F = 0$ on $\Sigma$.
\end{enumerate}
\end{corollary}

\begin{proof}
    Notice that the system in Lemma~\ref{lem:bndryctrl} is satisfied under these assumptions.  Hence, it must follow that 
    \[\int_\Sigma \left(2F-HF_H\right)d\mu = 0.\]
    Since $\mathcal{W}$ is either shrinking or expanding, this implies that $2F-HF_H = 0$ pointwise on $\Sigma$.  If $H$ is not constant on $\Sigma$, then by continuity this equation is satisfied for an open set in $H$, so either $F \equiv cH^2$ for some $c$ or $F \equiv 0$ on $\Sigma$ by analyticity.  Otherwise, $\Sigma$ has constant mean curvature, and $F = 0$ on $\partial\Sigma$ implies that $F \equiv 0$ on $\Sigma$.
\end{proof}

\begin{remark}
If $F$ is assumed to be smooth instead of analytic, the conclusions of Corollary~\ref{cor:minimal} remain true only in a local sense.  That is, either $F = cH^2$ or $F = 0$ pointwise on $\Sigma$, but this need not extend to the whole domain of $\mathcal{W}$.  To see this, take for example $F(H) = \varphi(H)\, H^2$, where $\varphi(H)$ is a smooth bump function which is identically 1 on $[-1,1]$ and supported on $[-\sqrt{2}, \sqrt{2}]$ (see \cite[Chapter 13]{tu2010} for a construction).  In this case, one can verify that $2F - HF_H \geq 0$ everywhere, so $\mathcal{W}$ is expanding, and also that all derivatives of $\varphi$ vanish on $[-1,1]$.  This means that the Clifford torus $\Sigma \subset \mathbb{R}^3$ which has been rescaled so that its mean curvature lies in $[-1,1]$ is critical for $\mathcal{W}$, and (vacuously) satisfies the boundary conditions in Corollary~\ref{cor:minimal}.  However, $\mathcal{W}$ is certainly not scaling invariant nor identically zero on its domain.  
\end{remark}

This particular Corollary can be used to show that, in some cases, minimizers of a $\mathcal{W}$-functional can only be minimal surfaces.  In particular, for a surface with boundary $\Sigma \subset \mathbb{R}^3$ there is the 
notion of p-Willmore energy mentioned in the Introduction,
\begin{equation*}
    \mathcal{W}^p(\*r) = \int_{\Sigma} |H|^p\, d \mu \qquad p\in\mathbb{R}.
\end{equation*}
Clearly, this coincides up to a constant factor with the usual, scaling invariant, definition of the (non-conformal) Willmore energy when $p=2$.  On the other hand, this functional is not scaling invariant for $p\neq 2$, since
\begin{equation*}
    2F - HF_H - 2KF_K = 2|H|^p - H\, \partial_H\left((H^2)^{p/2}\right) = (2-p)|H|^p \neq 0.
\end{equation*}
As seen before, this lack of scaling invariance has significant consequences on the critical surfaces of $\mathcal{W}^p$.  In particular, we observe that conditions on the boundary of a p-Willmore surface when $p>2$ exert much more control over what happens in the interior when compared to the case $p=2$.  To illustrate this, first note that the flux formula (\ref{eq:scaleflux}) reduces immediately to
\begin{equation}\label{eq:pwillflux}
\begin{split}
    &(2-p)\int_\Sigma |H|^p\, d\mu \\
    &= \int_{\partial\Sigma} \Big(|H|^p\langle \*r,\con\rangle + p|H|^{p-2}\left(H\nabla_\con\langle \*r,\*n\rangle - (p-1)\langle \*r,\*n\rangle \nabla_\con H \right)\Big)\, ds.
\end{split}
\end{equation}

Corollary~\ref{cor:minimal} can now be applied to establish the statement of Theorem~\ref{thm:pWillmore}.

\begin{proof}[Proof of Theorem~\ref{thm:pWillmore}]
By Corollary~\ref{cor:minimal}, it is sufficient to consider 
\[|H|^p = p|H|^{p-2}H = p(p-1)H^{p-2}\nabla_\con H = 0 \quad\text{on}\,\,\partial\Sigma,\]
which is clearly satisfied under the hypothesis that $p>2$ and $H = 0$ on the boundary.  Since $|H|^p$ is not scaling invariant for $p\neq 2$, it follows that $|H|^p \equiv 0$ on $\Sigma$.  Hence, $H\equiv 0$ and $\Sigma$ must be minimal.
\end{proof}

\begin{remark}
This result can also be deduced directly from (\ref{eq:pwillflux}) without appealing to Corollary~\ref{cor:minimal}, as the flux formula (\ref{eq:pwillflux}) reduces to
    \begin{equation*}
        \int_\Sigma |H|^p = 0,
    \end{equation*}
implying that $|H|\equiv 0$ on $\Sigma$ by continuity.
\end{remark}


Clearly this is quite different from the Willmore case of $p=2$, where there are many known and non-minimal solutions to the same boundary-value problem (e.g. \cite{deckelnick2009}).  It is likely true that other $\mathcal{W}$-functionals which lack scale-invariance are similarly influenced by their boundary data, but this is a study for future work.  It is hoped that the results and Corollaries developed here will be of use in answering such questions.